\documentclass[11pt]{article}

\usepackage{amsmath, amsfonts, amssymb, amsxtra, amsthm, mathrsfs}

\usepackage[utf8]{inputenc}
\usepackage{breqn}
\usepackage{tikz}
\usepackage{mathdots}
\usepackage{cancel}
\usepackage{siunitx}
\usepackage{array}
\usepackage{multirow}
\usepackage{tabularx}
\usepackage{extarrows}
\usepackage{booktabs}
\usetikzlibrary{fadings}
\usetikzlibrary{patterns}
\usetikzlibrary{shadows.blur}
\usetikzlibrary{shapes}
\usepackage{amscd}
\usepackage{graphicx, color}
\usepackage{caption}
\usepackage{t1enc}
\usepackage[mathscr]{eucal}
\usepackage{indentfirst}
\usepackage{graphics}
\usepackage{pict2e}
\usepackage{epic}
\numberwithin{equation}{section}
\usepackage[margin=2.9cm]{geometry}
\usepackage{epstopdf} 
\usepackage{float}

\usepackage{blindtext}
\usepackage{titlesec}

\theoremstyle{plain}
\newtheorem{Thm}{Theorem}[section]
\newtheorem{Lem}[Thm]{Lemma}
\newtheorem{Cor}[Thm]{Corollary}
\newtheorem{Prop}[Thm]{Proposition}

 \theoremstyle{definition}

\newtheorem{?}[Thm]{Problem}

\newcommand{\sys}{\operatorname{sys}}
\newcommand{\volwp}{\operatorname{vol_{WP}}}
\newcommand{\Mg}{\mathcal{M}_g}
\usepackage{etoolbox}
\makeatletter
\patchcmd{\l@section}
  {\hfil}
  {\leaders\hbox{\normalfont$\m@th\mkern \@dotsep mu\hbox{.}\mkern \@dotsep mu$}\hfill}
  {}{}

 \usepackage{tocloft}

\cftpagenumbersoff{part}

\title{Intersection Number, Length, and Systole on Compact Hyperbolic Surfaces}
 \author{Tina Torkaman}

\begin{document}




\date{}

 \maketitle



\thispagestyle{empty}

Published in \emph{Geometriae Dedicata}.
\bigskip

\begin{abstract}

The interaction strength $I(X)$ of a compact hyperbolic surface $X$ is the best upper bound for the intersection number of two closed geodesics divided by the product of their lengths. Let $\mathcal{M}_g$ be the moduli space of compact hyperbolic surfaces of genus $g$ and $\sys(X)$ the length of a shortest closed geodesic on $X \in \mathcal{M}_g$. We determine the asymptotic behavior of $I(X)$, as $X \to \infty$ in $\mathcal{M}_g$, in terms of $\sys(X)$. We also determine the approximate behavior of the minimum of $I(X)$ over $\mathcal{M}_g$, as $g \to \infty$.

\end{abstract}

\tableofcontents

\section{Introduction}
\label{intro}

%
%

Bounds on the intersection number of closed curves on surfaces have been studied in various settings \cite{Basmj}\cite{Chas.Phillips.pair.pants}\cite{Chas.Phillips.punc.torus}\cite{Curt.Moi}\cite{cheboui}. Let $\mathcal{M}_{g,n}$ be the moduli space of finite-volume hyperbolic surfaces of genus $g\geq 2$ with $n$ cusps. When $n=0$ we write $\mathcal{M}_g$. In this paper, we study the optimal upper bound on the (geometric) intersection number of closed geodesics on $X \in \mathcal{M}_{g,n}$ in terms of the product of their lengths.  

\textbf{Intersection number.} Given the closed geodesics $\gamma_1,\gamma_2$ in $X$, let $i(\gamma_1,\gamma_2)$ denote {\em the (geometric) intersection number} of $\gamma_1,\gamma_2$, which is the number of transversal intersection points (counted with multiplicity) between them. The intersection number is an invariant of the homotopy classes $[\gamma_1], [\gamma_2]$; in particular, it is independent of the geometry of $X$. Let $\ell_X(\gamma)$ denote the length of $\gamma$. We know $i(\gamma_1,\gamma_2)/(\ell(\gamma_1)\ell(\gamma_2))$ is bounded above by a constant depending on $X \in \mathcal{M}_g$ \cite{Basmj}.   
 We define

\begin{equation}\label{equ: I}
    I(X):= \sup \limits_{\gamma_1,\gamma_2} \frac{i(\gamma_1,\gamma_2)}{\ell_X(\gamma_1)\ell_X(\gamma_2)},
\end{equation}
where the supremum is over all pairs of closed geodesics.
We refer to $I(X)$ as {\em the interaction strength} of $X$, since it controls the best upper bound on $i(\gamma_1,\gamma_2)$ in terms of $\ell_X(\gamma_1)\ell_X(\gamma_2)$. Let $\mathcal{G}_X$ be the set of all closed geodesics and $\sys(X):=\min \limits_{\gamma \in \mathcal{G}_X} \ell_X(\gamma)$ denote the systole of $X$. 

\textbf{Asymptotic behavior.} When $X \in \mathcal{M}_g$, it is easy to see that $I(X)\to \infty$ as $X \to \infty$ in $\Mg$, see \S \ref{sec: systole}, and 
$$
I(X)\leq 4/\sys(X)^2,
$$
see Proposition \ref{prop: int.bnd}.
                 
 It is well known that $X \to \infty$ in $\Mg$ if and only if $\sys(X)\to 0$ \cite{Mumf}. 
Our main result describes the exact asymptotic behavior of $I(X)$ as a function on $\Mg$ in terms of $\sys(X)$.

\begin{Thm}\label{theorem: main.order}  
For $g\geq 2$, we have: 
$$
I(X)\sim \frac{1}{2\sys(X)\log(1/\sys(X))},
$$
as $X \to \infty$ in $\Mg$.
\end{Thm}

The notation $\sim$ means that their ratio tends to $1$.

Let $\gamma_1$ be a closed geodesic with length $\sys(X)$. We can find a closed geodesic $\gamma_2$ that intersects $\gamma_1$ in two points and whose length is at most about $4\log(1/\sys(X))$. The normalized intersection number $i(\gamma_1,\gamma_2)/(\ell_X(\gamma_1)\ell_X(\gamma_2))$ is a lower bound for $I(X)$ and is at least about $ 1/(2\sys(X)\log(1/\sys(X)))$; see Lemma \ref{lemma: geod.int.sys} and Figure \ref{fig: sys(X).main}.

\begin{figure}[H]
    \centering

\tikzset{every picture/.style={line width=0.75pt}} 

\begin{tikzpicture}[x=0.75pt,y=0.75pt,yscale=-1,xscale=1]

\draw    (90.2,107.2) .. controls (143.2,76.2) and (162.2,160.2) .. (301.2,153.2) ;
\draw    (75.2,145.2) .. controls (79.1,190.38) and (116.73,194.06) .. (171.2,184.2) .. controls (225.67,174.34) and (268.9,161.72) .. (320.2,177.2) ;
\draw    (301.2,153.2) .. controls (481.2,144.2) and (509.2,100.2) .. (512.2,170.2) ;
\draw    (320.2,177.2) .. controls (332.41,180.75) and (343.02,184.29) .. (355.36,188.03) .. controls (367.7,191.78) and (503.62,225.11) .. (512.2,170.2) ;
\draw    (75.2,145.2) .. controls (74.2,130.2) and (80.2,116.2) .. (90.2,107.2) ;
\draw    (420.2,176.2) .. controls (441.4,188.4) and (469.2,188.2) .. (488,171) ;
\draw    (91,147) .. controls (104.2,159.2) and (123.2,166.2) .. (162.2,148.2) ;
\draw    (102,154) .. controls (116.2,144.2) and (142.2,146.2) .. (149.2,154.2) ;
\draw    (433.2,181.2) .. controls (446.2,176.2) and (466.2,168.2) .. (477,178) ;
\draw [color={rgb, 255:red, 208; green, 2; blue, 27 }  ,draw opacity=1 ]   (287.3,170.03) .. controls (277.67,170.8) and (278.67,152.8) .. (285.3,154.03) ;
\draw [color={rgb, 255:red, 208; green, 2; blue, 27 }  ,draw opacity=1 ]   (149.2,154.2) .. controls (148.42,167.18) and (205.75,162.45) .. (269.19,159.01) .. controls (342.16,155.05) and (423.21,152.8) .. (433.2,181.2) ;
\draw [color={rgb, 255:red, 208; green, 2; blue, 27 }  ,draw opacity=1 ] [dash pattern={on 0.84pt off 2.51pt}]  (149.2,154.2) .. controls (150.76,136.86) and (210.81,146.42) .. (275.31,159.29) .. controls (339.81,172.15) and (433.3,201.03) .. (433.2,181.2) ;
\draw [color={rgb, 255:red, 208; green, 2; blue, 27 }  ,draw opacity=1 ] [dash pattern={on 0.84pt off 2.51pt}]  (285.3,154.03) .. controls (292.53,153.15) and (292.53,168.15) .. (287.3,170.03) ;

\draw (266,130.4) node [anchor=north west][inner sep=0.75pt]  [font=\normalsize]  {$\gamma_1 $};
\draw (416,150.4) node [anchor=north west][inner sep=0.75pt]  [font=\normalsize]  {$\gamma_2 $};

\end{tikzpicture}

   \caption{}
    \label{fig: sys(X).main}
\end{figure}

Since $I(X)$ is continuous (see Lemma \ref{lemma: I.cont}), from Theorem \ref{theorem: main.order}, we obtain an estimate of $I(X)$ up to a multiplicative constant depending on $g$ for all surfaces in $\Mg$ (see Corollary \ref{cor: I.asymp.c_g}). 
 
 The notation $A \asymp B$ $( A \asymp_g B)$ means $A \in [B/c,cB]$ $(A \in [B/c_g,c_gB])$ for an implicit constant $c>0$ $(c_g>0)$.  We say that $A$ and $B$ are asymptotically equivalent when $A \sim B$ and have the same order of magnitude when $A \asymp B$.

Moreover, $I(X)$ is a proper continuous function, according to Theorem \ref{theorem: main.order} and Lemma \ref{lemma: I.cont}, so it attains its minimum at a point in $\Mg$. The order of magnitude of the minimum of $I$ is as follows. 

\begin{Prop}\label{Prop: I.min}
We have:
$$
\min \limits_{X \in \Mg} I(X) \asymp \frac{1}{(\log g)^2}, 
$$
uniformly in $g$.
\end{Prop}

\textbf{Finite volume case.} Let $X \in \mathcal{M}_{g,n}$. When $n\geq 1$, $I(X)=\infty$. In fact, the intersection number can be exponentially large in terms of length. For example, a closed geodesic $\gamma_n$ that turns around a cusp $n$ times has length $\ell_X(\gamma_n)\sim 2\log n$ and $i(\gamma_n,\gamma_n) \sim 2n$ for $n \in \mathbb{N}$. 
However, when closed geodesics are restricted to a compact subset of $X$, we obtain a result similar to Theorem \ref{theorem: main.order}. Given a subsurface $Y\subset X$, define  
$$
I(Y) := \sup \limits_{\gamma_1,\gamma_2 \subset Y} \frac{i(\gamma_1,\gamma_2)}{\ell_X(\gamma_1)\ell_X(\gamma_2)},
$$
where the supremum is over all pairs of closed geodesics in $Y$. Let $X_r$ be the compact subset of $X$ obtained by removing, for each cusp, the standard horoball bounded by a horocycle of length $r<1$. 

\begin{Thm}\label{thrm: finite.vol}
 Fix $g>1, n>0$, and let $X$ be a finite-volume hyperbolic surface of genus $g$ with $n$ cusps. For $s= \min(\sys(X),1/2)$, we have:
  $$
  I(X_r)\sim \max \left ( \frac{1}{2s\log(1/s)}, \frac{1}{r(\log(1/r))^2} \right),
 $$
 as $\min(s,r) \to 0$.
\end{Thm}

This also gives a global estimate of $I(X_r)$ up to a multiplicative constant for all $X \in \mathcal{M}_{g,n}$ (see Corollary \ref{cor: finite.vol}).

\textbf{Complements.} Here are some related results.
\begin{itemize}
    \item One can also consider the following similar quantities: 
    $$
    I_{\Delta}(X):=\sup_{\gamma} \frac{i(\gamma,\gamma)}{\ell_X(\gamma)^2},
    $$
    where the supremum is over all closed geodesics and
    $$
    I_{simple}(X):= \sup_{\gamma_1,\gamma_2} \frac{i(\gamma_1,\gamma_2)}{\ell_X(\gamma_1)\ell_X(\gamma_2)},
    $$
    where the supremum is over all pairs of simple closed geodesics.
    
     We will see that these have the same order of magnitude as $I(X)$. More precisely, we have:
     $$
     I_{\Delta}(X) \in [I(X)/2,I(X)],
     $$ 
     see Proposition \ref{prop: I_delta=I}, and
     
     $$
     I_{simple}\asymp_g I(X),
     $$
     see Corollary \ref{cor: I_simple=I}. Moreover, $I_{simple}(X)/I(X) \to 1$ as $X \to \infty$ in $\Mg$.
     
    \item Although $I(X) \to \infty$ as $X \to \infty$ in $\Mg$, the expected interaction strength with respect to the Weil–Petersson measure is still finite. In other words, we have:
    
    \begin{Cor}\label{cor: exp}
    The expected value of the interaction strength on $\Mg$, equipped with the \textit{Weil-Petersson} measure, is finite.
    \end{Cor}
    
    This is an immediate consequence of Theorem \ref{theorem: main.order} and Mirzakhani's result about the volume of $\Mg^{\epsilon}$, the $\epsilon$-thin part of $\Mg$  \cite[Thm. 4.1]{Mir.WP}. Note that the upper bound $4/\sys(X)^2$ is insufficient to obtain Corollary \ref{cor: exp}. See \S \ref{sec: random} for more details.
     
\end{itemize}
 \textbf{Outline of the proof of Theorem \ref{theorem: main.order}.} The main steps of the proof are as follows. In order to find the lower bound for $I(X)$, we consider a systole, $i.e.$, a closed geodesic with the shortest length and a shortest closed geodesic that intersects it at two points (as we explained when we described Figure \ref{fig: sys(X).main}). In order to find the required upper bound, we consider a \emph{proper pants decomposition} $\mathcal{P}$ of $X$, that is, a pants decomposition in which the length of every curve is less than the Bers' constant (a constant depending on $g$); see \S \ref{sec: prem}. We then split $X$ into thin and thick parts, where the thin parts are annular neighborhoods of the curves in $\mathcal{P}$ (they are contained in their collars), and the thick parts are the complements of the thin parts (see \S\ref{sec: prem}). The following inequality for variables $I_i,\ell_i,\ell'_i \geq 0, \, i=1,2$ :

$$
\frac{I_1+I_2}{(\ell_1+\ell_2)(\ell_1'+\ell_2')} \leq \max(\frac{I_1}{\ell_1\ell_1'},\frac{I_2}{\ell_2\ell_2'}),
$$
implies that it is enough to obtain the required upper bound for each part separately, and we will see that the thin part(s) containing the systole(s) have the main contribution to $I(X)$. The heart of the proof is to approximate the lengths and intersection numbers of geodesic arcs inside a thin part (see \S\ref{sec: proof}).  

 \textbf{Questions.}
 \begin{itemize}
     \item Where is the supremum attained in Equation (\ref{equ: I}) (the definition of $I(X)$)?\\
 We can see $I(X)$ is achieved by a pair of geodesic currents; see Corollary \ref{cor: I=2godesic}. One may ask whether $I(X)$ is achieved by a pair of closed geodesics, particularly simple closed geodesics. For instance, is it true that for almost every $X\in \Mg$ there are simple closed geodesics $\gamma_1,\gamma_2$ such that $i(\gamma_1,\gamma_2)=I(X)\ell_X(\gamma_1)\ell_X(\gamma_2)?$ If it is true, then we would have $I_{simple}(X)=I(X)$ for all $X\in \Mg$.
 
 The work of Thurston in the article \cite{Th1} inspires this question (which can also be found in the book \cite{Thurstonbook}).
     Let $\mathcal{T}_g$ be the Teichmüller space of hyperbolic surfaces of genus $g$. Given a closed curve $\gamma$, define $\ell_X(\gamma)$ as the hyperbolic length of the geodesic representative of the homotopy class $[\gamma]$ on $X$. Thurston proved that for generic pairs $X, Y \in T_g$, the stretch factor $L(X, Y):=\sup \limits_{\gamma} \log  (\ell_X(\gamma)/\ell_Y(\gamma))$, where the supremum is over all the homotopy classes of closed curves, is attained by a simple closed geodesic; in other words,
     $L(X,Y)=\log(\ell_X(\gamma)/\ell_Y(\gamma))$ for a simple closed geodesic $\gamma$ \cite[\S 10 ]{Th1}.
     \item Is $I_{simple}(X)=I(X)$ for all $X\in \Mg$? Note that $I_{simple}(X)\leq I(X)$ clearly, and by Corollary \ref{cor: I_simple=I}, they have the same asymptotic behavior.
     \item Is $I_{\Delta}=I(X)/2$ for all $X\in \Mg$?
     \end{itemize}

\textbf{Notes and references.} The relation between self-intersection number and word length on topological surfaces is also studied in  \cite{Chas.Phillips.pair.pants}\cite{Chas.Phillips.punc.torus}\cite{Curt.Moi}. A quantity similar to the interaction strength for algebraic intersection numbers is studied in \cite{Kvol1}\cite{Kvol2}. The statistic of the self-intersection number is studied in \cite{lalley.stat} for hyperbolic surfaces and in \cite{chas.lalley.stat} for topological surfaces. There are also some results in \cite{Basmj} about the infimum of the length of a closed geodesic when its self-intersection number is fixed. Interestingly, $\sys(X)\log(1/\sys(X))$ also appears in the volume of the unit ball of the measured laminations \cite{Herrara.b(X)}.

\textbf{Acknowledgements.} I would like to thank Curt T. McMullen for his continuous help, invaluable discussions, and insightful suggestions related to the results of this paper. I would also like to thank Yongquan Zhang and Rafael M. Saavedra for their helpful comments. Finally, I am grateful to the anonymous referee for a careful reading of the manuscript and for valuable suggestions.


\section{Hyperbolic geometry}\label{sec: prem}
This section briefly explains some basics of hyperbolic geometry (for a reference, see \cite{Basmj}). We also present some properties of the function $\sigma(x)=\sinh^{-1}(1/\sinh(x))$ which is half of the length of the collar of a simple closed geodesic $\gamma$ with $\ell_X(\gamma)=2x$. At the end of this section, we prove $I_{\Delta}\asymp I$ using the theory of geodesic currents. Note that throughout this paper, $c_0,c'_0 (c_g,c'_g)$ refer to positive constants (depending only on $g$). They are not necessarily any fixed, specific value. For simplicity, we may write $\ell(-)$ for length, instead of $\ell_X(-)$. Systole may also refer to one of the shortest closed geodesics on $X$.

Let $(X,d)$ be a hyperbolic surface (with metric $d$) and $\pi:\mathbb{H} \to X$ its universal cover map, where $\mathbb{H}$ is the upper half-plane. We denote the length of an arc $\alpha$ by $\ell(\alpha)$. Let $\mathcal{G}_X$ be the set of closed geodesics of $X$.

\textbf{Collar of a closed geodesic.}
Assume that $\gamma$ is a simple closed geodesic on $X$. Let $N_{\gamma}(r):=\{ x \in X| \,  d(x,\gamma) \leq r \}$ be the $r-$neighborhood of $\gamma$. 
\begin{Lem}\label{lemma: area-length}
Assume that $N_{\gamma}(r)$ is an embedded annulus in $X$. Then its boundary components have length $\ell(\gamma)\cosh(r)$ and $\gamma$ splits $N_{\gamma}(r)$ into two parts of area $\ell(\gamma) \sinh(r)$.
\end{Lem}
\begin{proof}
Apply the Fermi coordinate with respect to $\gamma$, which is:
$$
ds^2=d\rho^2+\cosh(\rho)^2dt^2,
$$
where $\rho$ is the distance from $\gamma$ and $t$ is the distance between the projection at $\gamma$ and a based point $P_0 \in \gamma$ \cite[p. 4]{Bsr}. 
\end{proof}

Define 
$$
\sigma(\gamma):=\sigma\!\left(\tfrac{\ell(\gamma)}{2}\right), \quad 
\sigma(x)=\sinh^{-1}\!\left(\tfrac{1}{\sinh(x)}\right).
$$
For a simple closed geodesic $\gamma$, the neighborhood $N_{\gamma}(r)$ is called the \emph{collar of $\gamma$} when $r=\sigma(\gamma)$, and is denoted by $O(\mathrm{collar},\gamma)$. It is well known that the collar of a simple closed geodesic is embedded in $X$ \cite[Prop.~3.1.8]{Bsr}.

Next, consider a pair of pants $S$, i.e., a sphere with three boundary components. The boundary curves of $S$ are called the \emph{cuffs} of $S$. When $S$ is equipped with a hyperbolic metric with geodesic boundary, the shortest geodesic arcs that connect distinct cuffs are called \emph{seams} of $S$.

Every hyperbolic surface $X$ admits a \emph{pants decomposition}, obtained by cutting $X$ along a maximal collection of disjoint simple closed geodesics. Note that a pair of pants in this decomposition may have boundary components of length $0$ if $X$ has cusps.

\textbf{Thin-thick decomposition.} Let $X \in \mathcal{M}_g$. Consider a pants decomposition $\mathcal{P}=\{P_1,\dots,P_{2g-2}\}$ of $X$ by closed geodesics $\gamma_1,\dots,\gamma_{3g-3}$. The collars of curves $\gamma_k$ are pairwise disjoint annuli \cite[Prop. 3.1.8.]{Bsr}. Let $O(collar, k)$ be the collar of $\gamma_k$. Let $O(\mathrm{thin},k)$ denote the region 
$$
O(\mathrm{thin},k) := N_{\gamma_k}(r_k) \subset O(\mathrm{collar},k),
$$ 
called the {\em thin part}, where 
$$
r_k:=\sinh^{-1} \left( \frac{1}{\sinh(\ell(\gamma_k/2))\sqrt{1+\cosh^2(\ell(\gamma_k)/2)}} \right ).
$$  
  
We emphasize that here the terminology '\textit{thin part}' is used in a non-standard sense: it refers to the chosen subsets $N_{\gamma_k}(r_k)$ of collars, not to collars around short geodesics (as in the classical thick–thin decomposition).

The complements of the thin parts are $2g-2$ pairs of pants, which are \emph{thick parts} of $X$ with respect to $\mathcal{P}$. 
Denote the thick part in $P_j$ by $O(thick,j)$.  

The motivation behind defining the thin parts as above is to avoid very short geodesic arcs in a collar whose endpoints are on the boundary of the collar. This is mandatory for our proof. We will show that there is a lower bound for the length of a geodesic arc which has a nonempty intersection with the thin part $O(thin,\gamma)$ and whose endpoints are on $\partial O(collar,\gamma)$; see Lemma \ref{lemma: has.int} and Lemma \ref{lemma: length.type2}. 

From Bers' theorem, we know that there is a pants decomposition such that $\ell(\gamma_k) \leq L_g$, where $L_g$ is the \textit{Bers' constant} \cite[Thm. 5.1.2.]{Bsr}. We call such a decomposition a \textit{proper decomposition}. From now on, we assume that $\mathcal{P}$ is proper.

 \begin{Lem}\label{lemma: thick.part}
 Consider the thick part $O(thick,j)$ and the seams of $P_j$. \\(a) The length of any geodesic arc in $O(thick,j)$ with endpoints on the seams is $\geq c_g$.\\
 (b) The distances between the cuffs of $O(thick,j)$ are $\geq c_g$.
 \end{Lem}

\begin{proof}
For part (a), assume that $AB$ is a geodesic arc in $O(thick,j)$ with endpoints $A$ and $B$ on the seams. Let $A', B'$ be the projection of $A$ and $B$ to $\gamma_k$. Note that $\ell(AA'), \ell(BB') \geq r_k$ and it is well known that $\ell(A'B')=\ell(\gamma)/2$.  
From the hyperbolic geometry formula for the quadrilateral $ABB'A'$ with two right angles  (see \cite[Equ. 2.3.2.]{Bsr}) we know:
 $$
 \cosh(\ell(AB))=\cosh(\ell(AA'))\cosh(\ell(BB'))\cosh(\ell(A'B'))-\sinh(\ell(AA'))\sinh(\ell(BB')).
 $$
 Therefore, we have:
 $$
 \cosh(\ell(AB)) \geq \cosh(\ell(AA'))\cosh(\ell(BB'))(\cosh(\ell(A'B'))-1) \geq 2\cosh(r_k)^2\sinh^2(\ell(A'B')/2).
 $$
 We used the fact that $\cosh x=2\sinh^2 (x/2)+1$. After simplifying the right-hand side, we obtain:
 $$
 \cosh(\ell(AB)) \geq 2(1+\frac{1}{\sinh^2(\ell(\gamma_k/2)(1+\cosh^2(\ell(\gamma_k)/2))})\sinh^2(\ell(\gamma_k)/4)
 $$
 $$
 \geq \frac{2\sinh^2(\ell(\gamma_k)/4)}{\sinh^2(\ell(\gamma_k)/2)(1+\cosh^2(L_g))}\geq c'_g,
 $$
as required. 

For part $(b)$, it is enough to show that the distance between the boundaries of $O(thin,k)$ and $O(collar,k)$ is $\geq c_g$ (the boundaries that are on the same side of $\gamma_k$); see Figure \ref{fig: proof.d}. We have the following: 

$$
d(l):=distance=\sigma(\gamma_k)-r_k=\sinh^{-1}(\frac{1}{\sinh(l)})-\sinh^{-1}(\frac{1}{\sinh(l)\sqrt{1+\cosh(l)^2}}),
$$
where $l=\ell(\gamma_k)/2$. See Figure \ref{fig: proof.d}.

\begin{figure}[H]
    \centering
    
\tikzset{every picture/.style={line width=0.75pt}} 

\begin{tikzpicture}[x=0.75pt,y=0.75pt,yscale=-1.2,xscale=1.1]

\draw    (105,129) .. controls (220.2,131.2) and (243.2,101.2) .. (269.2,71.2) ;
\draw    (106,159) .. controls (146.2,158.2) and (285.2,148.2) .. (327.2,192.2) ;
\draw    (105,129) .. controls (118.42,128.8) and (118.42,157.8) .. (106,159) ;
\draw    (105,129) .. controls (98.2,131.2) and (98.2,160.2) .. (106,159) ;
\draw    (206.27,117.8) .. controls (225.27,111.8) and (225.07,158.51) .. (212.2,159.2) ;
\draw  [dash pattern={on 4.5pt off 4.5pt}]  (206.27,117.8) .. controls (194.27,126.8) and (198.07,156.51) .. (212.2,159.2) ;
\draw    (256.27,162.8) .. controls (279.1,163.22) and (263.1,86.22) .. (248,94) ;
\draw  [dash pattern={on 4.5pt off 4.5pt}]  (248,94) .. controls (234.1,108.22) and (244.1,158.22) .. (256.27,162.8) ;
\draw    (269.2,71.2) .. controls (283.72,90.83) and (304.72,112.83) .. (344.72,105.83) ;
\draw    (269.2,71.2) .. controls (301.72,65.83) and (339.72,87.83) .. (344.72,105.83) ;
\draw    (327.2,192.2) .. controls (364.72,180.83) and (367.72,156.83) .. (367.72,140.83) ;
\draw    (327.2,192.2) .. controls (313.72,177.83) and (344.72,131.83) .. (367.72,140.83) ;
\draw    (344.72,105.83) .. controls (297.72,122.83) and (328.72,138.83) .. (367.72,140.83) ;
\draw    (95.12,141.23) -- (80.62,128.55) ;
\draw [shift={(79.12,127.23)}, rotate = 41.19] [color={rgb, 255:red, 0; green, 0; blue, 0 }  ][line width=0.75]    (10.93,-3.29) .. controls (6.95,-1.4) and (3.31,-0.3) .. (0,0) .. controls (3.31,0.3) and (6.95,1.4) .. (10.93,3.29)   ;
\draw  [color={rgb, 255:red, 208; green, 2; blue, 27 }  ,draw opacity=1 ][line width=0.75] [line join = round][line cap = round] (105.16,123.76) .. controls (105.16,117.27) and (116.72,121.48) .. (122.16,118.76) .. controls (123.02,118.33) and (126.73,111.76) .. (128.16,111.76) .. controls (129.23,111.76) and (129.55,114.58) .. (131.16,114.76) .. controls (133.83,115.06) and (136.5,115.38) .. (139.16,115.76) .. controls (146.41,116.8) and (156.1,114.49) .. (164.16,113.76) .. controls (176.8,112.61) and (188.52,110.42) .. (200.16,108.76) .. controls (205.19,108.04) and (208.16,107.81) .. (208.16,111.76) ;
\draw  [color={rgb, 255:red, 208; green, 2; blue, 27 }  ,draw opacity=1 ][line width=0.75] [line join = round][line cap = round] (108.16,166.76) .. controls (114.2,172.8) and (124.58,168.97) .. (132.16,172.76) .. controls (133.65,173.51) and (135.16,176.76) .. (135.16,176.76) .. controls (135.16,176.76) and (134.25,176.22) .. (134.16,175.76) .. controls (133.66,173.26) and (135.38,168.96) .. (138.16,168.76) .. controls (153.61,167.66) and (168.46,172.43) .. (183.16,172.76) .. controls (197.82,173.09) and (212.49,172.44) .. (227.16,172.76) .. controls (234.4,172.92) and (256.16,180.33) .. (256.16,169.76) ;
\draw    (116,143) .. controls (187.81,138.83) and (230.81,144.83) .. (266.81,128.83) ;
\draw  [color={rgb, 255:red, 208; green, 2; blue, 27 }  ,draw opacity=1 ][line width=0.75] [line join = round][line cap = round] (222.19,136.52) .. controls (222.19,129.71) and (226.91,132.22) .. (233.19,131.52) .. controls (233.88,131.44) and (235.19,128.52) .. (235.19,128.52) .. controls (235.19,128.52) and (238.19,131.19) .. (242.19,130.52) .. controls (247.66,129.61) and (253.4,123.48) .. (258.19,122.52) .. controls (260.7,122.02) and (264.19,121.46) .. (264.19,124.52) ;
\draw  [color={rgb, 255:red, 208; green, 2; blue, 27 }  ,draw opacity=1 ][line width=0.75] [line join = round][line cap = round] (119.19,143.52) .. controls (125.63,137.07) and (139.02,138.38) .. (150.19,137.52) .. controls (154.19,137.21) and (158.19,136.88) .. (162.19,136.52) .. controls (163.91,136.36) and (163.18,133.51) .. (163.19,133.52) .. controls (163.57,134.66) and (163.99,136.39) .. (165.19,136.52) .. controls (174.14,137.48) and (183.21,136.27) .. (192.19,135.52) .. controls (196.93,135.13) and (219.19,135.3) .. (219.19,140.52) ;
\draw  [color={rgb, 255:red, 208; green, 2; blue, 27 }  ,draw opacity=1 ][line width=0.75] [line join = round][line cap = round] (116.23,147.39) .. controls (116.23,151.2) and (123.03,150.04) .. (126.23,150.39) .. controls (137.95,151.7) and (147.18,150.23) .. (158.23,148.39) .. controls (168.43,146.69) and (178.9,147.05) .. (189.23,147.39) .. controls (189.62,147.41) and (194.23,153.39) .. (194.23,153.39) .. controls (194.23,153.39) and (191.03,147.53) .. (196.23,147.39) .. controls (208.9,147.05) and (221.94,149.47) .. (234.23,146.39) .. controls (242.37,144.36) and (265.23,143.76) .. (265.23,134.39) ;

\draw (132,182) node [anchor=north west][inner sep=0.75pt]  [font=\footnotesize] [align=left] {collar};
\draw (276,135.4) node [anchor=north west][inner sep=0.75pt]  [font=\footnotesize]  {$ \begin{array}{l}
thick\ part\\
\end{array}$};
\draw (67,106.4) node [anchor=north west][inner sep=0.75pt]    {$\gamma _{k}$};
\draw (112,94.4) node [anchor=north west][inner sep=0.75pt]  [font=\footnotesize]  {$thin\ part$};
\draw (230,112.4) node [anchor=north west][inner sep=0.75pt]    {$d(l)$};
\draw (162,117.4) node [anchor=north west][inner sep=0.75pt]  [font=\small]  {$r_{k}$};
\draw (177,154.4) node [anchor=north west][inner sep=0.75pt]  [font=\footnotesize]  {$\sigma ( \gamma _{k})$};

\end{tikzpicture}

    \caption{}
    \label{fig: proof.d}
\end{figure}

On the one hand, $d(l)$ is a continuous and positive function of $l \in (0, L_g/2]$ (since $\sinh$ is continuous and strictly increasing). On the other hand, a straightforward calculation shows that
$$
d(l) \to \tfrac{1}{2} \log 2 \quad \text{as } l \to 0.
$$
Hence, $d(l)$ is bounded below by a positive constant depending only on $g$.

\end{proof}

 The following lemma describes an essential property of thin parts and justifies the definition of $r_k$.
 
  \begin{Lem}\label{lemma: has.int}
 Let $\alpha$ be a geodesic arc in the collar of $\gamma$ with endpoints on one boundary of the collar. Assume that $\alpha$ does not intersect itself; then it is disjoint from the interior of $O(thin,\gamma)$.
 \end{Lem}

\begin{proof}

Consider preimages $\widetilde{\gamma}$ and $\widetilde{\alpha}$ of $\gamma$ and $\alpha$ in $\mathbb{H}$, respectively. The endpoints $A,B$ of $\widetilde{\alpha}$ have distance $\sigma(\gamma)$ from $\widetilde{\gamma}$.
Let $A'$ and $B'$ be the projection of $A,B$ on $\widetilde{\gamma}$. Therefore, $\ell(AA')=\ell(BB')=\sigma(\gamma)$ and $\ell(A'B')\leq \ell(\gamma)$ since $\alpha$ doesn't intersect itself. See Figure \ref{fig: no.intersect.alpha} and Figure \ref{fig: lemma.inthick}.

\begin{figure}[H]
    \centering
    
\tikzset{every picture/.style={line width=0.75pt}} 
\begin{tikzpicture}[x=0.75pt,y=0.75pt,yscale=-1,xscale=1]
\draw    (100,120) .. controls (152.2,120.52) and (235.2,144.52) .. (306.2,66.52) ;
\draw    (107,165) .. controls (160.2,160.52) and (282.2,158.52) .. (324.2,184.52) ;
\draw    (306.2,66.52) .. controls (348.2,75.52) and (341.2,184.52) .. (324.2,184.52) ;
\draw    (306.2,66.52) .. controls (292.2,73.52) and (292.2,171.52) .. (324.2,184.52) ;
\draw    (100,120) .. controls (117.2,120.52) and (122.2,162.52) .. (107,165) ;
\draw    (100,120) .. controls (88.2,125.52) and (88.2,161.52) .. (107,165) ;
\draw [color={rgb, 255:red, 208; green, 2; blue, 27 }  ,draw opacity=1 ]   (297,96) .. controls (240.2,111.52) and (264.2,169.52) .. (304.2,159.52) ;
\draw  [dash pattern={on 4.5pt off 4.5pt}]  (114.2,132.52) .. controls (163.2,142.52) and (253.2,131.52) .. (297,96) ;
\draw  [dash pattern={on 4.5pt off 4.5pt}]  (116.2,153.52) .. controls (165.2,148.52) and (262.2,152.52) .. (304.2,159.52) ;
\draw (298,85.4) node [anchor=north west][inner sep=0.75pt]  [font=\small]  {$A$};
\draw (308,150.4) node [anchor=north west][inner sep=0.75pt]  [font=\small]  {$B$};
\draw (109,121.4) node [anchor=north west][inner sep=0.75pt]  [font=\footnotesize]  {$A'$};
\draw (102,142.4) node [anchor=north west][inner sep=0.75pt]  [font=\footnotesize]  {$B'$};

\end{tikzpicture}

    \caption{}
    \label{fig: no.intersect.alpha}
\end{figure}

Let $m_{\alpha}$ and $m_{\gamma}$ be the midpoints of $AB$ and $A'B'$. We can see that $m_{\alpha}m_{\gamma}$ is orthogonal to $AB$, $A'B'$ as follows. Consider the isometry $f$ which is a reflection with respect to $\ell$ the geodesic orthogonal to $\tilde{\gamma}$ at the midpoint $m_{\gamma}$. Map $f$ sends $\tilde{\gamma}$ to itself, $A'$ to $B'$, and so $A$ to $B$. Therefore, $f$ maps $AB$ to itself and the midpoint $m_{\alpha}$ is on $\ell$, as required. The length of $m_{\alpha}m_{\gamma}$ gives the distance between $\widetilde{\alpha}$ and $\widetilde{\gamma}$ (see Figure \ref{fig: lemma.inthick}).  

\begin{figure}[H]
    \centering

\tikzset{every picture/.style={line width=0.75pt}} 

\begin{tikzpicture}[x=0.75pt,y=0.75pt,yscale=-0.9,xscale=0.9]

\draw    (144,224) -- (501,225) ;
\draw    (222.5,224.5) .. controls (222,109) and (403,109) .. (406,223) ;
\draw  [color={rgb, 255:red, 0; green, 0; blue, 0 }  ][line width=3] [line join = round][line cap = round] (227,123) .. controls (229.69,123) and (227.26,125.77) .. (227,125) .. controls (226.67,124) and (226.95,122) .. (228,122) ;
\draw  [color={rgb, 255:red, 0; green, 0; blue, 0 }  ][line width=3] [line join = round][line cap = round] (379,117) .. controls (380,117) and (381,120.91) .. (381,118) ;
\draw    (228,123) .. controls (256,74) and (354,75) .. (380,119) ;
\draw    (228,123) .. controls (242,129) and (251,144) .. (254,155) ;
\draw    (363,150) .. controls (363,138) and (367,125) .. (380,119) ;
\draw    (299.2,85.6) .. controls (306.2,101) and (310.73,120.9) .. (309.2,136.6) ;
\draw   (301.66,136.87) -- (301.34,127.86) -- (308.88,127.6) ;
\draw   (300.67,92.05) -- (293.56,93.68) -- (292.09,87.22) ;
\draw   (251.39,149.46) -- (256.07,147.25) -- (258.68,152.79) ;
\draw  [line width=3] [line join = round][line cap = round] (254.05,155.55) .. controls (254.05,155.22) and (254.05,154.89) .. (254.05,154.55) ;
\draw  [line width=3] [line join = round][line cap = round] (363.14,149.01) .. controls (362.81,149.01) and (362.48,149.01) .. (362.14,149.01) ;
\draw   (364.03,142.82) -- (359.05,140.54) -- (355.98,147.24) ;
\draw  [line width=0.75] [line join = round][line cap = round] (232.4,115.47) .. controls (235.42,115.47) and (237.02,123.84) .. (234.4,126.47) ;

\draw (207,94.4) node [anchor=north west][inner sep=0.75pt]    {$A$};
\draw (387,104.4) node [anchor=north west][inner sep=0.75pt]    {$B$};
\draw (231,148.4) node [anchor=north west][inner sep=0.75pt]    {$A'$};
\draw (368,139.4) node [anchor=north west][inner sep=0.75pt]    {$B'$};
\draw (340,75.4) node [anchor=north west][inner sep=0.75pt]    {$\tilde{\alpha }$};
\draw (329,116.4) node [anchor=north west][inner sep=0.75pt]    {$\tilde{\gamma }$};
\draw (290,67.4) node [anchor=north west][inner sep=0.75pt]    {$m_{\alpha }$};
\draw (296.81,140.14) node [anchor=north west][inner sep=0.75pt]    {$m_{\gamma }$};
\draw (239,110.4) node [anchor=north west][inner sep=0.75pt]    {$\phi $};

\end{tikzpicture}

    \caption{}
    \label{fig: lemma.inthick}
\end{figure}
 
  
  Now we show that the distance between $\alpha$ and $\gamma$ is $\geq r_k$.
 Let $\phi$ be the angle between $AB$ and $AA'$. Define: $m:=\ell(m_{\gamma}m_{\alpha})$, $a:=\ell(AA')=\sigma(\gamma)$, and $b:=\ell(A'm_{\gamma})$. From the hyperbolic formula for the trirectangle $Am_{\alpha}m_{\gamma}A'$ \cite[Thm. 2.3.1.]{Bsr} 
 we know $\sinh m=\cos \phi /\sinh b, \tan \phi=\coth b/\sinh a$, therefore, we have: 
 
$$
\sinh{m}=\frac{\cos \phi}{\sinh{b}}=\frac{1}{\sinh{b}\sqrt{1+\tan^2{\phi}}}=\frac{1}{\sinh{b}\sqrt{1+(\coth{b}/\sinh{a})^2}}=\frac{\sinh{a}}{\sqrt{\sinh^2{a}\sinh^2{b}+\cosh^2{b}}}.
$$

Therefore, from $b=\ell(A'B')/2 \leq \ell(\gamma)/2$ we have:
$$
\sinh{m} \geq \frac{\sinh{a}}{\sqrt{\sinh^2{a}\sinh^2(\ell(\gamma)/2)+\cosh^2(\ell(\gamma)/2)}}=\frac{1}{\sinh(\ell(\gamma)/2)\sqrt{1+\cosh^2(\ell(\gamma)/2)}}.
$$
As a result, $m \geq r_k$, and $\gamma$ is disjoint from the interior of $O(thin,\gamma)$.
 
\end{proof}

 \textbf{Intersection number.} Given closed curves $\alpha,\beta$, the \emph{(geometric) intersection number} of them, $i(\alpha, \beta)$, is the minimal number of (transversal) intersection points between representatives of homotopy classes $[\alpha], [\beta]$ and the points are counted with multiplicity. Geodesic representatives always attain the minimum. By multiplicity, we mean that we count a point with weight if the curves pass through it multiple times. To find the weight at a point $p$, we can slightly move the curves in a neighborhood of $p$ to have only simple intersections. Then, the minimum number of simple intersections is the weight. We can bound the intersection number above in terms of the product of the lengths.
  
 \begin{Prop}\label{prop: int.bnd}
 Let $X$ be a compact hyperbolic surface. Then we have:
 $$
 i(\alpha,\beta) \leq \frac{4}{\sys(X)^2}\ell_X(\alpha)\ell_X(\beta),
 $$
 for all $\alpha,\beta \in \mathcal{G}_X$.
 \end{Prop}
 \begin{proof}

We first claim that any two geodesic arcs $\alpha_1,\alpha_2$ of length at most $\sys(X)/2$ intersect transversally at most once. Indeed, suppose $P$ is an intersection point. Then both arcs are contained in the disk of radius $\sys(X)/2$ centered at $P$. If they intersected at more than one point, they would bound a geodesic bigon inside this disk. Such a bigon cannot exist on a hyperbolic surface, by the Gauss-Bonnet theorem.  

Now let $\alpha,\beta$ be two closed geodesics. Consider the product space
$$
\alpha \times \beta = \{(a,b)\mid a \in \alpha,\, b \in \beta \},
$$
which is isometric to a flat torus with the product metric $d\ell_1 \times d\ell_2$, where $d\ell_i$ denotes the hyperbolic length along $\alpha$ or $\beta$. In particular, the total area is
$$
area(\alpha \times \beta) = \ell_X(\alpha)\,\ell_X(\beta).
$$

Each intersection point $p \in \alpha \cap \beta$ determines a unique pair $(a,b)\in \alpha\times \beta$ with $a=b=p$. Around this point, consider the square of side length $\sys(X)/2$ in the product metric, i.e.
$$
\Bigl\{(a',b') \in \alpha \times \beta \;\Big|\; d(a,a'),
d(b,b') \leq \tfrac{\sys(X)}{4}\Bigr\}.
$$

Moreover, since two short arcs cannot intersect more than once, the squares corresponding to distinct intersection points are disjoint inside $\alpha \times \beta$. Thus, the total area of $\alpha \times \beta$ must be at least the sum of the areas of these squares:
$$
i(\alpha,\beta)\cdot \frac{\sys(X)^2}{4}\;\; \leq \;\; \ell_X(\alpha)\,\ell_X(\beta).
$$

 \end{proof}
 
Define $I(X)$, the \emph{interaction strength} of $X$, as the optimal constant $c(X)$ such that $i(\gamma_1,\gamma_2)\leq c(X) \ell(\gamma_1)\ell(\gamma_2)$ for all $\gamma_1,\gamma_2 \in \mathcal{G}_X$. In other words,
$$
I(X):= \sup \limits_{\gamma_1,\gamma_2 \in \mathcal{G}_X} \frac{i(\gamma_1,\gamma_2)}{\ell(\gamma_1)\ell(\gamma_2)}.
$$
From Proposition \ref{prop: int.bnd}, we know $I(X)$ is finite.

\begin{Lem}\label{lemma: I.cont}
The function $I: \mathcal{M} \to \mathbb{R}$ is continuous.
\end{Lem}
\begin{proof}
The Teichmüller space $\mathcal{T}_g$ is the universal cover of $\Mg$. Therefore, it is equivalent to show $I(X)$ is continuous on $\mathcal{T}_g$ which is a consequence of the fact that the function $\ell_{\gamma}: \mathcal{T}_g \to \mathcal{R}^+$ is continuous where $\gamma$ is a closed curve and $\ell_{\gamma}(X)$ is the length of the geodesic representative in the homotopy class $[\gamma]$ on $X$. More precisely, we observe that for any $X \in \mathcal{T}_g$ and $\delta>0$, there exists an open neighborhood $U_{\delta}$ of $X$ such that for every closed curve $\gamma$, we have
$$
\ell_Y(\gamma) \in (\frac{\ell_X(\gamma)}{1+\delta},\, (1+\delta)\ell_X(\gamma))
\quad \text{for all } Y\in U_{\delta}
$$ 
(see, for example stretch factor in \cite{Th1}). On the other hand, the intersection number between two closed curves depends only on their homotopy classes and not on their length. Consequently, for any given $\epsilon>0$ if $\delta$ is chosen sufficiently small compared to it, then
$$
I(Y) \in (\frac{I(X)}{1+\epsilon},\, (1+\epsilon) I(X)),
$$
for all $Y \in U_{\delta}$, as required.

\end{proof} 
\textbf{Properties of $\boldsymbol{\sigma(t).}$}
 Now we discuss some properties of the function $\sigma(t)=\sinh^{-1}(\frac{1}{\sinh(t)})$ which are repeatedly applied in the proofs. We may not refer to this lemma when we apply it. 
 
\begin{Lem}\label{lemma: sigma}
Assume that $t>0$. For the function $\sigma(t)$ we have:
\begin{enumerate}
 \item $\sigma(t) \geq \log \frac{1}{t}$, 
 \item $\sigma(t) \leq 2\log \frac{1}{t} $ for $t \leq 1/3$,
 \item $\sigma(t)$ is decreasing,  
\item $t\sigma(t)$ is $<1$ and increasing for $t \leq \frac{1}{2}$, 
\item $\lim \limits_{t \to 0} \frac{\log(1/t)}{\sigma(t/2)}=1$.
\end{enumerate}
\end{Lem}

\begin{proof}
\emph{Part 1}. If $t\geq 1$, then $\sigma(t)>0>\log \frac{1}{t}$. So, assume $t \in (0,1]$. The function $\sinh(X)$ is increasing, therefore, it is enough to prove $\sinh(\sigma(t))\geq \sinh(\log(1/t))$, or equivalently, $(\frac{1}{t}-t)(e^t-e^{-t}) \leq 4$.
\begin{itemize}
    \item When $t \in [\log 2,1]$ we have:
$$
(\frac{1}{t}-t)(e^t-e^{-t})\leq 2te(\frac{1}{t}-t) \leq 2e(1-(\log 2)^2)< 4.
$$
Note that $(e^t-e^{-t})/(2t)=e^x\leq e$ for some $x \in [-t,t]$.

\item When $t \in (0,\log 2)$, then we have:
$$
(\frac{1}{t}-t)(e^t-e^{-t})\leq (\frac{1}{t}-t)(2te^{\log 2}) < 4.
$$
\end{itemize}
\emph{part 2.} Similar to the part $1$, it is equivalent to show 
$$
4 \leq (e^t-e^{-t})(\frac{1}{t^2}-t^2),
$$
which is true since we have:
$$
(e^t-e^{-t})(\frac{1}{t^2}-t^2) \geq 2te^{\frac{-1}{3}}(\frac{1}{t^2}-t^2) \geq 2e^{\frac{-1}{3}}(3-\frac{1}{27})>4. 
$$
In the first inequality, we used the fact that $e^t-e^{-t}=2te^y$ for a $y \in (-t,t)$, so $e^y\geq e^{-t}\geq e^{-1/3}$. In the second inequality, we used the fact that $1/t^2-t^2$ is decreasing.

\emph{Part 3.} It is directly implied from the fact that $\sinh(t)$ is increasing.

\emph{Part 4.} First we show $t\sigma(t)<1$. Note that $e^x-e^{-x} > 2x$ for $x \in \mathbb{R}^+$. Therefore, $\sinh(t)\sinh(1/t) > 1$ which implies $1/t > \sigma(t).$
In order to show $t\sigma(t)$ is increasing we prove its derivative is positive for $t \leq \frac{1}{2}.$
In other words, our aim is to prove
$$
\sigma(t)-\frac{t}{\sinh(t)} \geq 0,
$$
when $t \leq 1/2$. This is true since we have:
$$
\sigma(t) \geq \sigma(\frac{1}{2})\geq 1 \geq \frac{t}{\sinh(t)}.
$$

\emph{Part 5.} We have $\sinh^{-1}(N)/\log{2N} \to 1$ when $N \to \infty$. As a result, we obtain 
$$
\lim \limits_{t \to 0} \frac{\log( 1/t)}{\sigma(t/2)}=\lim \limits_{t \to 0} \frac{\log(1/t)}{\log( 2/\sinh(t/2))}.
$$
On the other hand, $\sinh(x) \in [x,2x]$ when $x>0$ is small enough. Therefore, we have:

$$
\frac{\log (1/t)}{\log(2/\sinh(t/2))} \in [\frac{\log(1/t)}{\log(4/t)},\frac{\log(1/t)}{\log(2/t)}].
$$
We can see that the bounds tend to $1$ as $t \to 0$. 
\end{proof}

\textbf{Geodesic currents.} A \emph{geodesic currents} is a Borel measure on $T_1(X)$, the unit tangent bundle of $X$, which is invariant under the geodesic flow and the involution map (which sends $(x,v) \in T_1(X)$ to $(x,-v)$). Let $\mathbb{P}(X)$ be the space of all tangent directions on $X$ and $\mathcal{F}$ the foliation of $\mathbb{P}(X)$ obtained from the geodesic flow. A geodesic current $C$ can also be defined as a transverse measure of $\mathcal{F}$. In other words, it assigns a measure to every plane transverse to the leaves of $\mathcal{F}$ such that the measures are invariant under the flow along the leaves of $\mathcal{F}$. To see the connection between these definitions, note that $C \times d\ell$ is a measure on $T_1(X)$, which is invariant under the geodesic flow and the involution map. 

Let $\mathcal{C}$ be the space of all geodesic currents of $X$. We dropped the dependence on $X$ since it is a topological space. This means that there is a canonical correspondence between the geodesic currents of two marked hyperbolic surfaces in $\mathcal{T}_g$, the Teichmüller space of marked complete hyperbolic surfaces with genus $g$. This correspondence sends homotopic closed geodesics to each other.

The space $\mathcal{C}$ is equipped with the weak topology obtained from the topology on the space of all measures on $T_1(X)$.
Each closed geodesic is an example of a geodesic current, and the set of all multi-geodesics is dense in space $\mathcal{C}$ (see \cite{Bon.gc.3} and \cite{Bon.gc.Tech} for more details about geodesic currents).

For $X \in \mathcal{M}_g$, the functions $\ell:\mathcal{G}_X \to \mathbb{R}^+$ and $i:\mathcal{G}_X\times \mathcal{G}_X \to \mathbb{R}^+$ extend continuously to $\mathcal{C}$ and $\mathcal{C}\times\mathcal{C}$, respectively (when $X$ is not compact $i(-,-)$ is not continuous anymore). The set of geodesic currents with length $1$ is compact \cite[Prop. 4.7]{Bon.gc.3}. As a result, $I(X)$ is attained by a pair of length $1$ geodesic currents.

\begin{Cor}\label{cor: I=2godesic}
    There are length one geodesic current $C_1,C_2$ such that $I(X)=i(C_1,C_2)$.
\end{Cor}

Recall that
$$
    I_{\Delta}(X):=\sup_{\gamma} \frac{i(\gamma,\gamma)}{\ell_X(\gamma)^2},
    $$
    where the supremum is over all closed geodesics.

\begin{Prop}\label{prop: I_delta=I}
 For $X \in \Mg$ we have $I_{\Delta}(X) \in [I(X)/2,I(X)]$.
 \end{Prop}
 \begin{proof}
 Clearly $I_{\Delta}(X)\leq I(X)$. Now assume $I(X)=i(C_1,C_2)$ where $C_1,C_2$ are geodesic currents of length $1$. Then $C_1+C_2$ is a geodesic current of length $2$. The intersection number is a bilinear function on $\mathcal{C}\times\mathcal{C}$. So, we have:
 $$
 I_{\Delta}(X)\geq \frac{i(C_1+C_2,C_1+C_2)}{4}\geq \frac{i(C_1,C_2)}{2}=\frac{I(X)}{2}.
 $$
  
 \end{proof}

\section{Systole}\label{sec: systole}
In this section, we prove the lower bound in Theorem \ref{theorem: main.order}, and later, from the proof, we conclude that $I_{simple}$ and $I$ are comparable (see Corollary \ref{cor: I_simple-asymp} and Corollary \ref{cor: I_simple=I}). 

Let $\gamma_1$ be a systole of $X$. It is known that $\gamma_1$ is simple.

\begin{Lem} \label{lemma: geod.int.sys}
There exists a simple closed geodesic $\gamma_2$ such that either 
\begin{itemize}
\item $i(\gamma_1,\gamma_2)=2$, $\ell(\gamma_2)\leq 4\log(1/\sys(X))+c_g$, or
\item $i(\gamma_1,\gamma_2)=1$, $\ell(\gamma_2)\leq 2\log(1/\sys(X))+c_g$.
\end{itemize}
\end{Lem}
\begin{proof}
The geodesic $\gamma_1$ splits $N_{\gamma_1}(r)$ into two halves $N_{\gamma_1}^1(r), N_{\gamma_1}^2(r)$. These are not homotopic to an (embedded) annulus when $r\geq \sinh^{-1}(\frac{4\pi(g-1)}{\ell(\gamma_1)})$, since in that case, by Lemma \ref{lemma: area-length}, the area of $N_{\gamma_1}^i(r)$ is $\geq area(X)$, for $i=1,2$. 

Let $r_1$ be the smallest $r$ such that $N_{\gamma_1}^1(r)$ is homotopic to an annulus with two identified points, call it $P$, on one of the boundaries. There might be cases where we have more identified points, but the proof would be similar. Let $e_1 \subset N_{\gamma_1}^1(r_1)$ be a simple geodesic arc with endpoints on $\gamma_1$, that passes through $P$, and that has length 
$$
\leq 2r_1 \leq 2\sinh^{-1}(\frac{4\pi(g-1)}{\sys(X)}).
$$
See Figure \ref{fig: lower.bnd1}. Similarly, we define $r_2$ and $e_2$ in $N_{\gamma_1}^2(r_2)$.

\begin{figure}[H]
    \centering
    
\tikzset{every picture/.style={line width=0.75pt}} 

\begin{tikzpicture}[x=0.75pt,y=0.75pt,yscale=-0.7,xscale=0.7]

\draw    (412,83) .. controls (531,63) and (601,212) .. (441,228) ;
\draw    (138,117) .. controls (223,206) and (337,87) .. (412,83) ;
\draw    (137,234) .. controls (186,171) and (371,237) .. (441,228) ;
\draw    (387,155) .. controls (427,130) and (432,139) .. (461,152) ;
\draw    (372,146) .. controls (405,171) and (432,176) .. (472,146) ;
\draw [line width=1.5]    (207,207) .. controls (222,181) and (215,163) .. (205,151) ;
\draw  [dash pattern={on 0.84pt off 2.51pt}]  (207,207) .. controls (188,200) and (188,162) .. (205,151) ;
\draw  [color={rgb, 255:red, 0; green, 0; blue, 0 }  ][line width=3] [line join = round][line cap = round] (387,155) .. controls (386.67,155.33) and (386.33,155.67) .. (386,156) ;
\draw    (329,114) .. controls (333,139) and (356,153) .. (387,155) ;
\draw  [dash pattern={on 0.84pt off 2.51pt}]  (329,114) .. controls (356,113) and (377,128) .. (387,155) ;
\draw    (358,222) .. controls (355,193) and (370,164) .. (387,155) ;
\draw  [dash pattern={on 0.84pt off 2.51pt}]  (387,155) .. controls (392,172) and (391,215) .. (358,222) ;
\draw [line width=1.5]    (216,173) .. controls (260,155) and (318,147) .. (387,155) ;
\draw  [dash pattern={on 0.84pt off 2.51pt}]  (193,186) .. controls (269,215) and (393,183) .. (387,155) ;

\draw (192,169.4) node [anchor=north west][inner sep=0.75pt]    {$\gamma_1 $};
\draw (305,137.4) node [anchor=north west][inner sep=0.75pt]    {$e_{1}$};
\draw (246,170.4) node [anchor=north west][inner sep=0.75pt]    {$N_{\gamma_1 }^{1}(r_1)$};

\end{tikzpicture}

    \caption{}
    \label{fig: lower.bnd1}
\end{figure}

 Connect the endpoints of $e_1$ and $e_2$ by disjoint subarcs of $\gamma_1$. Then we obtain a closed curve with length 
 $$
 <4.\sinh^{-1}(\frac{4\pi(g-1)}{\sys(X)})+\sys(X).
 $$
 Let $\gamma_2$ be its geodesic representative.
 
 Note that $i(\gamma_1,\gamma_2)=2$ and

$$
\ell(\gamma_2) \leq 4\log \frac{12\pi(g-1)}{\sys(X)}+c_g \leq 4\log \frac{1}{\sys(X)}+c'_g.
$$
 We used the fact that $\sinh^{-1}(x)<\log 3x$ when $x\geq 1$.\\
 
Now we show that $\gamma_2$ is homotopically non-trivial. Consider surface $S:=N_{\gamma_1}^1(r_1) \cup N_{\gamma_1}^2(r_2)$. It is a surface of genus $0$ with $4$ boundaries (or $n$ boundaries). Cutting $S$ along $\gamma_2$ splits $S$ into two pairs of pants. Therefore, $\gamma_2$ is nontrivial (see Figure \ref{fig: nontrivial}).

\begin{figure}[H]
    \centering
    
\tikzset{every picture/.style={line width=0.75pt}} 

\begin{tikzpicture}[x=0.75pt,y=0.75pt,yscale=-0.6,xscale=0.6]

\draw    (242,82) .. controls (278,138) and (343,150) .. (404,76) ;
\draw   (172,116.14) .. controls (167.38,106.69) and (179.31,91.37) .. (198.65,81.95) .. controls (217.98,72.52) and (237.39,72.54) .. (242,82) .. controls (246.61,91.46) and (234.68,106.77) .. (215.35,116.2) .. controls (196.02,125.62) and (176.61,125.6) .. (172,116.14) -- cycle ;
\draw   (404,76) .. controls (408.94,66.88) and (427,67.1) .. (444.35,76.49) .. controls (461.7,85.89) and (471.76,100.89) .. (466.83,110.01) .. controls (461.89,119.13) and (443.82,118.91) .. (426.47,109.52) .. controls (409.13,100.13) and (399.06,85.12) .. (404,76) -- cycle ;
\draw    (213,297) .. controls (261,244) and (245,170) .. (172,116.14) ;
\draw    (466.83,110.01) .. controls (378,196) and (414,247) .. (465,296) ;
\draw   (213,297) .. controls (217.63,286.97) and (235.61,285.42) .. (253.16,293.53) .. controls (270.71,301.63) and (281.18,316.33) .. (276.54,326.36) .. controls (271.91,336.39) and (253.93,337.94) .. (236.38,329.84) .. controls (218.84,321.73) and (208.37,307.03) .. (213,297) -- cycle ;
\draw   (403.61,329.63) .. controls (398.3,319.94) and (407.74,304.56) .. (424.7,295.27) .. controls (441.65,285.99) and (459.69,286.31) .. (465,296) .. controls (470.31,305.69) and (460.86,321.07) .. (443.91,330.35) .. controls (426.96,339.64) and (408.91,339.31) .. (403.61,329.63) -- cycle ;
\draw    (276.54,326.36) .. controls (302,251) and (389,257) .. (403.61,329.63) ;
\draw [line width=1.5]    (339,273) .. controls (370,215) and (358,159) .. (326,127) ;
\draw [line width=1.5]  [dash pattern={on 1.69pt off 2.76pt}]  (339,273) .. controls (316,261) and (296,176) .. (326,127) ;
\draw [line width=1.5]    (237,208) .. controls (287,230) and (379,232) .. (412,204) ;
\draw [line width=1.5]  [dash pattern={on 1.69pt off 2.76pt}]  (237,208) .. controls (287,186) and (374,188) .. (412,204) ;

\draw (269,220.4) node [anchor=north west][inner sep=0.75pt]    {$ \begin{array}{l}
\gamma_2 \\
\end{array}$};
\draw (351,141.4) node [anchor=north west][inner sep=0.75pt]    {$ \begin{array}{l}
\gamma_1 \\
\end{array}$};

\end{tikzpicture}

    \caption{}
    \label{fig: nontrivial}
\end{figure}

The constructed arcs $e_1,e_2$ are simple. Therefore, if they do not intersect, $\gamma_2$ is also simple, as required. But if $e_1,e_2$ intersect at a point like $Q \in X$, we consider the following simple closed geodesic instead. Choose subarcs $f_1 \subset e_1, \, f_2 \subset e_2$ from $Q$ to $\gamma_1$ such that $\ell(f_i) \leq \ell(e_i)/2$, for $i=1,2$. We can assume $f_1,f_2$ are not intersecting. Connect $f_1$ and $f_2$ by a subarc of $\gamma_1$ to obtain a simple closed curve. The intersection number of $\gamma_1$ and the constructed curve $\gamma_2$ is $1$. Hence, $\gamma_2$ is not homotopically trivial. Let $\widetilde{\gamma_2}$ be the geodesic representative of $\gamma_2$. We have $i(\gamma_1,\widetilde{\gamma_2})=1$ and $\ell(\widetilde{\gamma_2})\leq 2\sinh^{-1}(4\pi(g-1)/\sys(X))+\sys(X)\leq 2\log(1/\sys(X))+c_g$, as required. 
\end{proof}

\textbf{Proof of Theorem \ref{theorem: main.order}, the lower bound.} Assume that $\sys(X)$ is small enough. Now consider $\gamma_1$ and $\gamma_2$ as defined in Lemma \ref{lemma: geod.int.sys}, we have: 
  
  $$
  2I(X)\sys(X)\log (1/\sys(X)) \geq  \frac{2i(\gamma_1,\gamma_2)}{\ell(\gamma_1)\ell(\gamma_2)}\sys(X)\log (1/\sys(X)) \geq \frac{\log(1/\sys(X))}{\log(1/\sys(X))+c_g},
  $$
  
  which tends to $1$ when $\sys(X) \to 0$, as required.
  \qed

\section{Interaction strength in thin/thick part}\label{sec: proof}
In this section, we complete the proof of Theorem \ref{theorem: main.order}. Namely, we obtain an upper bound on $I(X)$ asymptotically equivalent to $1/(2\sys(X)\log(1/\sys(X)))$. To establish Theorem \ref{theorem: main.order}, we begin by stating the intermediate results on which its proof relies. After presenting the proof of Theorem \ref{theorem: main.order}, we return to complete the proof of these results in detail.

 
\textbf{Overview.} Let $\mathcal{P}$ be a proper decomposition of $X$ into the pairs of pants $P_1,\dots,P_{2g-2}$, and let $i_Y(-,-)$ and $\ell_Y(-)$ be the restriction of the intersection number and the length, respectively, to the subset $Y \subset X$. In other words, $\ell_Y(a)$ is equal to $\ell(a \cap Y)$ and $i_Y(a,b)$ is the number of the intersection points in $Y$. Similarly, let $i_{thin}(.,.), \, i_{thick}(.,.), \, \ell_{thin}(.), \, \ell_{thick}(.)$ be the intersection numbers and lengths in the thin parts and thick parts, respectively. Moreover, $\ell_{collar}(.)$ refers to the length in the collar.

For $\eta,\zeta \in \mathcal{G}_X$ we have:
\begin{equation}\label{equ: thin/thick}
\frac{i(\eta,\zeta)}{\ell(\eta)\ell(\zeta)}= \frac{i_{thick}(\eta,\zeta)}{\ell(\eta)\ell(\zeta)}+\frac{i_{thin}(\eta,\zeta)}{\ell(\eta)\ell(\zeta)} \leq \frac{i_{thick}(\eta,\zeta)}{\ell_{thick}(\eta)\ell_{thick}(\zeta)}+\frac{i_{thin}(\eta,\zeta)}{\ell_{collar}(\eta)\ell_{collar}(\zeta)}. 
\end{equation}

As we explained before, considering arcs inside the collar but counting their intersection numbers inside the thin part (instead of the collar) helps us to avoid very short arcs in the collar.

The following inequality, for variables $a_i,b_i, I_i\geq 0$, tells us that it is enough to find the upper bound for the interaction strength in the thin and thick parts separately:
\begin{equation}\label{equi: algebra}
\frac{I_1+I_2}{(a_1+a_2)(b_1+b_2)}\leq \max( 
 \frac{I_1}{a_1b_1}, \frac{I_2}{a_2b_2}).
 \end{equation}

 In other words, we have: 
\begin{equation} \label{equ: thin/thick1}
\frac{i(\eta,\zeta)}{\ell(\eta)\ell(\zeta)} \leq \max \limits_{1 \leq j \leq 2g-2} \frac{i_{O(thick,j)}(\eta,\zeta)}{\ell_{O(thick,j)}(\eta)\ell_{O(thick,j)}(\zeta)}+ \max \limits_{1\leq k \leq 3g-3} \frac{i_{O(thin,k)}(\eta,\zeta)}{\ell_{O(collar,k)}(\eta)\ell_{O(collar,k)}(\zeta)}.
\end{equation}

In the following, we approximate the length and the intersection number of arcs in the thin and thick parts to find the upper bound for each term.
 
\textbf{Thick part.}
The interaction strength is $\leq c_g$ in the thick parts.
\begin{Lem}\label{lemma: I.thick}
For all $\eta,\zeta \in \mathcal{G}_X$ we have:  
$$
\frac{i_{O(thick,j)}(\eta,\zeta)}{\ell_{O(thick,j)}(\eta)\ell_{O(thick,j)}(\zeta)}\leq c_g.
$$
\end{Lem}

\begin{proof}
Let $e_1,e_2,e_3$ and $d_1,d_2,d_3$ be the cuffs and seams of the thick part $O(thick,j)$, respectively.
 We can split $\eta \cap O(thick,j)$ and $\zeta \cap O(thick,j)$ into subarcs of the following types:
\begin{itemize}
    \item An arc between different cuffs that does not intersect the seams.
    \item An arc with endpoints on the cuffs or seams such that its interior intersects exactly one seam. 
\end{itemize}

Note that the subarcs may overlap, but a generic point is on at most two subarcs.  

Assume that $\eta\cap O(thick,j),\zeta \cap O(thick,j)$ are split into $\eta_1,\dots,\eta_n$ and $\zeta_1,\dots,\zeta_m$.  
On one hand, $\ell(\eta_s), \ell(\zeta_t)\geq c_g$, by Lemma \ref{lemma: thick.part}. On the other hand, $i(\eta_s,\zeta_t)\leq 2$ (to see this, we can split the pair of pants into two hexagons, and then it is easy to see that two geodesic arcs intersect at most once in each hexagon). Therefore, we have:

$$
\frac{i_{O(thick,j)}(\eta,\zeta)}{\ell_{O(thick,j)}(\eta)\ell_{O(thick,j)}(\zeta)} \leq \frac{4\sum \limits_{s,t}i(\eta_s,\zeta_t)}{(\sum \limits_{s}\ell(\eta_s))(\sum \limits_{t}\ell(\zeta_t))} \leq \frac{8}{c_g^2},
$$
as required. 
\end{proof}
Similar to Lemma \ref{lemma: I.thick}, when $\ell(\gamma_k) \geq c_0$, we can see that the contribution of the thin part $O(thin,k)$ is $<c'_0$. In other words, the term corresponding to $O(collar,k), O(thin,k)$ in the Equation \ref{equ: thin/thick1} is bounded above by a constant. Therefore, from now on, we can assume that $\ell(\gamma_k)$ is as small as required.

\textbf{Thin part.} In this part, we restrict the intersection number into the thin part $O(thin,k)$. But the lengths are restricted to the collar $O(collar,k)$ instead of $O(thin,k)$ to avoid very short arcs (with zero winding number around $\gamma_k$). In other words, we restrict our attention to $i_{O(thin,k)}(-,-)$ and $\ell_{O(collar,k)}(-)$.

\begin{Prop}\label{Prop: I.thin}
Assume that $\ell(\gamma_k) \leq c_0$. We have:
$$
\frac{i_{O(thin,k)}(\eta,\zeta)}{\ell_{O(collar,k)}(\eta)\ell_{O(collar,k)}(\zeta)} \leq \frac{1}{2\sys(X)\sigma(\sys(X)/2)}+c_g.
$$
\end{Prop}

\textbf{Proof of Theorem \ref{theorem: main.order}, the upper bound.} As we explained, it is enough to prove the appropriate upper bound
for the interaction strength in thin and thick parts separately. It is bounded above by a constant in thick parts. Proposition \ref{Prop: I.thin} gives the proper upper bound in thin parts. Note that:
$$
2\sys(X)\log(\frac{1}{\sys(X)}) \left( \frac{1}{2\sys(X)\sigma(\sys(X)/2)}+c_g \right) \to 1,
$$
as $\sys(X) \to 0$.
\qed 

\vspace{0.5em}
Now we aim to prove Proposition \ref{Prop: I.thin}.
We have two types of arcs in $O(collar,k)$ with the endpoints on its boundaries:

\begin{itemize}
    \item Type $1$ subarcs have endpoints on the different boundaries of the collar.  
    \item Type $2$ subarcs have endpoints on one single collar boundary. 
\end{itemize}

 Define {\em the winding number} $\omega(\alpha)$ of the geodesic arc $\alpha \subset O(collar,k)$ as follows. Let $\widetilde{\alpha}$ and $\widetilde{\gamma_k}$ be preimages of $\alpha$ and $\gamma_k$ in $\mathbb{H}$, respectively. Assume that $A,B$ are the endpoints of $\widetilde{\alpha}$ and $A',B'$ the projection of $A,B$ to $\widetilde{\gamma_k}$. Then $\omega(\alpha):= \lfloor  \ell(A'B')/\ell(\gamma_k)\rfloor$. If we have an arc $\alpha$ of type $2$, then $\omega(\alpha) \geq 1$, unless, by Lemma \ref{lemma: has.int}, it does not enter the thin part, which in that case, we do not need to consider $\alpha$ for the proof of Proposition \ref{Prop: I.thin}.

In the following, we estimate the length and the intersection number of arcs based on their types and winding numbers.

\begin{Lem}\label{lemma: length.type1}
Assume that $\alpha \subset O(collar,k)$ is a geodesic arc of type $1$. Then we have:
$$
\ell(\alpha) \geq \max(2\sigma(\gamma_k),2\sigma(\gamma_k)+(\omega(\alpha)+1)\ell(\gamma_k)-4).
$$

\end{Lem}
\begin{proof}
The distance between the boundaries of $O(collar,k)$ is $2\sigma(\gamma_k)$. Let $\widetilde{\gamma_k}$ and $\widetilde{\alpha}$ be the preimages of $\gamma_k$ and $\alpha$, respectively, in $\mathbb{H}$. 
The endpoints of the geodesic arc $\widetilde{\alpha}$ are $A$ and $B$, where $d(A,\widetilde{\gamma_k})=d(B,\widetilde{\gamma_k})=\sigma(\gamma_k)$. 
Let $A'$ and $B'$ be the projection of $A$ and $B$ on $\widetilde{\gamma_k}$, then the geodesic arc $A'B'$, has length $ \geq w(\beta)\ell(\gamma_k)$ (see Figure \ref{fig: proof.type1}).  \\

\begin{figure}[H]
    \centering

\tikzset{every picture/.style={line width=0.75pt}} 

\begin{tikzpicture}[x=0.75pt,y=0.75pt,yscale=-0.8,xscale=0.8]

\draw    (144,224) -- (501,225) ;
\draw    (222.5,224.5) .. controls (222,109) and (403,109) .. (406,223) ;
\draw  [color={rgb, 255:red, 0; green, 0; blue, 0 }  ][line width=3] [line join = round][line cap = round] (228,121) .. controls (230.69,121) and (228.26,123.77) .. (228,123) .. controls (227.67,122) and (227.95,120) .. (229,120) ;
\draw    (228,123) .. controls (280.8,100.6) and (335.8,131.6) .. (349.8,180.6) ;
\draw    (228,123) .. controls (242,129) and (251,144) .. (254,155) ;
\draw    (349.8,180.6) .. controls (352.6,168.27) and (357.6,161.27) .. (368.6,152.27) ;
\draw  [color={rgb, 255:red, 0; green, 0; blue, 0 }  ][line width=3] [line join = round][line cap = round] (349,181) .. controls (351.69,181) and (349.26,183.77) .. (349,183) .. controls (348.67,182) and (348.95,180) .. (350,180) ;
\draw   (251.79,150.07) -- (257.99,147.29) -- (260.2,152.21) ;
\draw   (362.6,156.27) -- (359.16,153.92) -- (362.1,149.61) ;
\draw  [line width=3] [line join = round][line cap = round] (254.87,155.18) .. controls (254.53,155.18) and (254.2,155.18) .. (253.87,155.18) ;
\draw  [line width=3] [line join = round][line cap = round] (367.87,152.18) .. controls (367.53,152.18) and (367.2,152.18) .. (366.87,152.18) ;

\draw (207,94.4) node [anchor=north west][inner sep=0.75pt]    {$A$};
\draw (353,178.4) node [anchor=north west][inner sep=0.75pt]    {$B$};
\draw (231,148.4) node [anchor=north west][inner sep=0.75pt]    {$A'$};
\draw (370,133.4) node [anchor=north west][inner sep=0.75pt]    {$B'$};
\draw (268,94.4) node [anchor=north west][inner sep=0.75pt]    {$\tilde{\alpha }$};
\draw (261,120.4) node [anchor=north west][inner sep=0.75pt]    {$\widetilde{\gamma _{k}}$};

\end{tikzpicture}
   \caption{}
    \label{fig: proof.type1}
\end{figure}

  From the hyperbolic geometry formula for the quadrilateral $ABB'A'$ with two right angles  (see \cite[Equ. 2.3.2.]{Bsr}) we have:
 $$
 \cosh{\ell(\widetilde{\alpha})}=\cosh(\sigma(\gamma_k))^2\cosh(\ell(A'B'))+\sinh(\sigma(\gamma_k))^2.
 $$
 As a result, we have:
 $$
 e^{\ell(\widetilde{\alpha})} \geq \cosh{\ell(\widetilde{\alpha})}\geq \frac{e^{2\sigma(\gamma_k)}}{4}\frac{e^{\omega(\alpha)\ell(\gamma_k)}}{2},
 $$
 therefore:
 $$
 \ell(\alpha)\geq 2\sigma(\gamma_k)+(\omega(\alpha)+1)\ell(\gamma_k)-\log{8}-1 \geq 2\sigma(\gamma_k)+(\omega(\alpha)+1)\ell(\gamma_k)-4.
 $$
Recall that we can assume $\ell(\gamma_k)$ is as small as we want, so we assume that it is less than $1$. 
\end{proof}

 \begin{Lem}\label{lemma: length.type2}
 Assume that $\alpha \in O(collar,k)$ is a geodesic arc of type $2$ and $\alpha \cap O(thin,k) \not= \emptyset$. Then we have:
$$
\ell(\alpha) \geq \max \, ( \, (w(\alpha)+1)\ell(\gamma_k)\, , \, 2\sqrt{(w(\alpha)+1)\ell(\gamma_k)\sigma(\gamma_k)} \, )
$$ 
when $\ell(\gamma_k)\leq c_0$.
\end{Lem}

\begin{proof}

Similar to the proof of Lemma \ref{lemma: length.type1} define $\widetilde{\gamma_k}$, $\widetilde{\alpha}, A, A', B, B'$. We know that $\ell(AA')=\ell(BB')=\sigma(\gamma_k)$ and $\ell(A'B')\geq w(\beta)\ell(\gamma_k)$ (see Figure \ref{fig: proof.type2}).  \\

\begin{figure}[H]
    \centering

\tikzset{every picture/.style={line width=0.75pt}} 

\begin{tikzpicture}[x=0.75pt,y=0.75pt,yscale=-0.8,xscale=0.8]

\draw    (144,224) -- (501,225) ;
\draw    (222.5,224.5) .. controls (222,109) and (403,109) .. (406,223) ;
\draw  [color={rgb, 255:red, 0; green, 0; blue, 0 }  ][line width=3] [line join = round][line cap = round] (227,123) .. controls (229.69,123) and (227.26,125.77) .. (227,125) .. controls (226.67,124) and (226.95,122) .. (228,122) ;
\draw  [color={rgb, 255:red, 0; green, 0; blue, 0 }  ][line width=3] [line join = round][line cap = round] (379,117) .. controls (380,117) and (381,120.91) .. (381,118) ;
\draw    (228,123) .. controls (256,74) and (354,75) .. (380,119) ;
\draw    (228,123) .. controls (242,129) and (251,144) .. (254,155) ;
\draw    (363,150) .. controls (367.07,137.87) and (372.07,127.87) .. (380,119) ;
\draw   (251.62,149.23) -- (255.69,147.55) -- (258.06,153.33) ;
\draw   (357.39,147.75) -- (360.16,140.83) -- (365.77,143.08) ;
\draw  [line width=3] [line join = round][line cap = round] (252.87,155.41) .. controls (253.2,155.41) and (253.53,155.41) .. (253.87,155.41) ;
\draw  [line width=3] [line join = round][line cap = round] (361.87,149.41) .. controls (362.2,149.41) and (362.53,149.41) .. (362.87,149.41) ;

\draw (207,94.4) node [anchor=north west][inner sep=0.75pt]    {$A$};
\draw (387,104.4) node [anchor=north west][inner sep=0.75pt]    {$B$};
\draw (225,148.4) node [anchor=north west][inner sep=0.75pt]    {$A'$};
\draw (368,139.4) node [anchor=north west][inner sep=0.75pt]    {$B'$};
\draw (278,68.4) node [anchor=north west][inner sep=0.75pt]    {$\tilde{\alpha }$};
\draw (301,114) node [anchor=north west][inner sep=0.75pt]    {$\widetilde{\gamma _{k}}$};

\end{tikzpicture}

   \caption{}
    \label{fig: proof.type2}
\end{figure}

  From the hyperbolic geometry formula for the quadrilateral $ABB'A'$ with two right angles  (see \cite[Equ. 2.3.2.]{Bsr}) we have:
 $$
 \cosh(\ell(\widetilde{\alpha}))=\cosh(\sigma(\gamma_k))^2\cosh(\ell(A'B'))-\sinh(\sigma(\gamma_k))^2,
 $$
 or equivalently:
 
 \begin{equation}\label{equ: A_n}
  \sinh(\ell(\widetilde{\alpha})/2)=\cosh(\sigma(\gamma_k))\sinh(\ell(A'B')/2). 
 \end{equation}
 We used the fact that $\cosh^2 x-\sinh^2 x=1$ and $\cosh x=2\sinh^2 x/2+1$.
 
 Note that $\ell(\alpha)>\sinh^{-1}(1)$, since, from Equation (\ref{equ: A_n}), we have: 
 $$
 \ell(\alpha) \geq 2\sinh^{-1}(\cosh(\sigma(\gamma_k))\sinh(\ell(\gamma_k)/2))=2\sinh^{-1}(\cosh(\ell(\gamma_k)/2)) \geq \sinh^{-1}(1).
 $$
 
As we explained before, we can assume $\ell(\gamma_k)$ is small enough, so $\cosh(\sigma(\gamma_k))$ can be assumed to be large enough. Therefore, from Equation (\ref{equ: A_n}), we conclude that
  $\sinh(\ell(\alpha)/2)\geq \sinh(\ell(A'B')/2+b_0)$ for a constant $\ell(\gamma_K) \leq b_0$. Now, we have:
  $$
  \ell(\alpha) \geq \ell(A'B')+2b_0 \geq \ell(\gamma_k)\omega(\alpha)+2b_0 \geq \ell(\gamma_k)(\omega(\alpha)+1).
  $$
 Now, we aim to prove the second lower bound in the statement of this lemma.

 Let $N_0$ be a sufficiently large constant to be chosen later.  
We can restrict to the case $w(\alpha)> N_0$; otherwise,
$$
2\sqrt{(w(\alpha)+2)\,\ell(\gamma_k)\,\sigma(\gamma_k)} 
   \;\leq\; 2\sqrt{(N_0+2)\,\ell(\gamma_k)\,\sigma(\gamma_k)}.
$$
By choosing $\ell(\gamma_k)$ small enough, we can make the right-hand side arbitrarily small.  
In particular, it can be arranged to be less than $\sinh^{-1}(1)$, which is $\leq \ell(\alpha)$, as required.

 Now consider the following cases:
 
 \begin{itemize}
     \item When $1/(w(\alpha)+1) \geq \ell(\gamma_k)/2$:
     from the inequalities $\cosh(\sigma(\gamma_k)) \geq \cosh(\log(2/\ell(\gamma_k)))\geq 1/\ell(\gamma_k) $, $\sinh(x)\geq x$, and Equation (\ref{equ: A_n}), we have:
     $$
\ell(\alpha) \geq 2\sinh^{-1}(w(\alpha)/2) \geq 2\log(w(\alpha)/2). 
     $$
     On the other hand, we know that $t\sigma(t)$ is increasing when $t$ is sufficiently small. Therefore, we have:
     $$
     \log(w(\alpha)/2) \geq \sqrt{4\log(w(\alpha)+1)} \geq \sqrt{2\sigma(1/(w(\alpha)+1))} \geq \sqrt{(w(\alpha)+1)\ell(\gamma_k)\sigma(\gamma_k)} 
     $$
     So, $\ell(\alpha) \geq 2\sqrt{(w(\alpha)+1)\ell(\gamma_k)\sigma(\gamma_k)}$. The first inequality holds because we have assumed $w(\alpha)>N_0$, where $N_0$ is chosen to be sufficiently large.

     \item When $1/(w(\alpha)+1) < \ell(\gamma_k)/2$:
     in this case, we have:
     $$
     e^{\ell(\alpha)/2}\geq \sinh{\frac{\ell(\alpha)}{2}}\geq \frac{e^{\sigma(\gamma_k)}}{2}\frac{e^{(\omega(\alpha)+1)\ell(\gamma_k)/2}}{c_0}.
     $$
     In the second inequality, we used the fact that $\sinh(x) \geq e^x/c_0$ for a constant $c_0$ when $x > 1$. Now, we obtain 
     $$
     \ell(\alpha) \geq 2\sigma(\gamma_k)+(\omega(\alpha)+1)\ell(\gamma_k)-c_0' \geq \sigma(\gamma_k)+(\omega(\alpha)+2)\ell(\gamma_k) \geq 2\sqrt{(\omega(\alpha)+1)\ell(\gamma_k)\sigma(\gamma_k)},
     $$
     
     as required.
 \end{itemize}
 
 \end{proof}
 
 Assume that $A_n, B_n$ represent geodesic arcs in $O(\text {thin},k)$ with winding numbers $n$ of types $1$ and $2$, respectively. 
 \begin{Lem}\label{lemma: int}
 The following table shows an upper bound on the intersection number between geodesic arcs based on their types and winding numbers. The indices $m$ and $n$ refer to the winding numbers.
 
 \begin{figure}[H]
     \centering

\tikzset{every picture/.style={line width=0.75pt}} 

\begin{tikzpicture}[x=0.75pt,y=0.75pt,yscale=-1,xscale=1]

\draw   (81.8,39.52) -- (290.8,39.52) -- (290.8,190.52) -- (81.8,190.52) -- cycle ;
\draw   (81.8,39.52) -- (150.15,39.52) -- (150.15,190.52) -- (81.8,190.52) -- cycle ;
\draw    (218.15,40.52) -- (218.15,190.52) ;
\draw    (81.8,90.3) -- (290.8,91.3) ;
\draw    (81.8,139.3) -- (290.8,140.3) ;

\draw (104,105.4) node [anchor=north west][inner sep=0.75pt]    {$A_{n}$};
\draw (103,156.4) node [anchor=north west][inner sep=0.75pt]    {$B_{n}$};
\draw (164,53.4) node [anchor=north west][inner sep=0.75pt]    {$A_{m}$};
\draw (240,54.4) node [anchor=north west][inner sep=0.75pt]    {$B_{m}$};
\draw (145,104.4) node [anchor=north west][inner sep=0.75pt]  [font=\scriptsize]  {$ \begin{array}{l}
\leq n+m+2\\

\end{array}$};
\draw (235,109.4) node [anchor=north west][inner sep=0.75pt]  [font=\scriptsize]  {$\leq m+1$};
\draw (164,159.4) node [anchor=north west][inner sep=0.75pt]  [font=\scriptsize]  {$\leq n+1$};
\draw (212,156.4) node [anchor=north west][inner sep=0.75pt]  [font=\scriptsize]  {$ \begin{array}{l}
\leq 2\min(n,m)\\
+2
\end{array}$};
\draw (90,48) node [anchor=north west][inner sep=0.75pt]  [font=\small] [align=left] {type of\\subarc};

\end{tikzpicture}

   \caption{Upper bounds on the intersection number}
     \label{fig: int.table}
 \end{figure}
  
 \end{Lem}
\begin{proof}

Recall that geodesic arcs have the minimum intersection number among the curves in their homotopy classes with fixed endpoints. Therefore, it is enough to prove the bounds for some representatives in their homotopy classes.
The universal cover of the thin part is an infinite rectangle $R=[-\infty,\infty]\times[0,1]$ with a deck transformation $(x,y) \to (x+1,y)$. Hence, the rectangle between $x=k$ and $x=k+1$ is a fundamental domain. Let $\pi$ be the universal covering map, which is a projection from $R$ to the thin part.

In the following, we describe the proper representatives in homotopy classes. Without loss of generality, assume that $n \leq m$.  
\begin{itemize}
    \item We can apply a Dehn twist around $\gamma_k$ on the thin part, and it does not change the intersection number between the arcs. When we have two arcs of type $A_n, A_m$, we can apply some Dehn twists to obtain arcs of types $A_0$ and $A_{m+n}$(or $A_{m-n}$). Assume that the arcs are line segments $P_1P_2$ and $Q_1Q_2$ in $R$, where the coordinates of the points are $P_1=(0,0), P_2=(c_0,1), Q_1=(0,0), Q_2=(m+n+c_0',1)$, for some constants $c_0,c_0' \in [0,1)$. Now, it is easy to see that the projections of these line segments to the thin part intersect at $\leq n+m+2$ points. Since $\pi(P_1P_2)$ intersects the segment $\pi(Q_1Q_2\cap [x,x+1]\times [0,1])$ at most once.

    \item Similarly, when we have arcs of type $A_n,B_m$, we may instead consider arcs of type $A_0,B_m$ by applying some Dehn twists around $\gamma_k$. The arc of type $B_m$ is a curve $Q_1Q_2$ that starts at the point $Q_1=(0,0)$ and ends at the point $Q_2=(m+c_0',0)$ for some $c_0' \in [0,1]$. Let the line segment $P_1P_2$ be the arc of type $A_0$, where $P_1=(0,0)$ and $P_2=(c_0,1)$ for $c_0 \in [0,1)$. Now, it is easy to see that if the curve is chosen properly, then the number of intersections between the projection of these arcs to the thin part is $\leq m+1$. Since $\pi(P_1P_2)$ intersects $\pi(Q_1Q_2 \cap [x,x+1]\times [0,1])$ at most once.
    
    \item Assume that we have arcs of type $B_n,B_m$. We may assume that the arc of type $B_n$ is in the lower half $((-\infty,\infty)\times[0,1/2])$ of $R$, and the other arc is constructed by three subarcs as follows. The middle subarc is in the upper half of $R$ and the end subarcs are vertical lines from $y=0$ to $y=1/2$. Then, similar to the previous case, we can see that each vertical line intersects the arc of type $B_n$ in $\leq n+1$ points. Therefore, their intersection number is $\leq 2n+2$.  
    
\end{itemize}

A more geometric argument for the case $i(A_n,A_m)$ can be found in \cite[Lemma 3.6]{Tina-equi}. That proof uses a pictorial approach by representing the proper homotopy class on the surface, and the same reasoning can be adapted to the other cases as well.
\end{proof}

\textbf{Proof of Proposition \ref{Prop: I.thin}.} 
Let $\eta_0=\eta \cap O(collar,k)$, $\zeta_0=\zeta\cap O(collar,k)$, $\omega(\eta_0)=n$, and $\omega(\zeta_0)=m$. Without loss of generality, we assume $n \leq m$. Define: 

$$
 I_0=\frac{i_{O(thin,k)}(\eta_0,\zeta_0)}{\ell_{O(collar,k)}(\eta)\ell_{O(collar,k)}(\zeta)}.
 $$
   Now, we apply the above lemmas to bound $I_0$ based on the types of the arcs.

\begin{itemize}
\item when $\eta_0,\zeta_0$ are of type one, we have:
$$
 I_0 \leq \frac{m+n+2}{\ell(\eta_0)\ell(\zeta_0)}.
 $$
 From Lemma \ref{lemma: length.type1}, we have:
 $$
 \ell(\eta_0) \geq 2\sigma(\gamma_k)+(n+1)\ell(\gamma_k)-4, 
\, \, \, \,  \ell(\zeta_0) \geq 2\sigma(\gamma_k)+(m+1)\ell(\gamma_k)-4.
$$

We aim to prove that $\ell(\eta_0)\ell(\zeta_0)$ is $\geq 2\ell(\gamma_k)\sigma(\gamma_k)(m+n+2)$. Consider the following cases. We obtain the proper lower bound in each case separately.

\begin{enumerate}
    
    \item If $n+1 \geq 4(m+1)/\sigma(\gamma_k)^2$ then
    $$
    \ell(\eta_0)\ell(\zeta_0)\geq 2\sigma(\gamma_k)(m+n+2)\ell(\gamma_k)+e,
     $$
    where 
    $$e=4\sigma(\gamma_k)^2+(n+1)(m+1)\ell(\gamma_k)^2-4\ell(\gamma_k)(m+n+2)-16 \sigma(\gamma_k). 
     $$
     We aim to show $e\geq 0$. Since $m \geq n$, we have
    \begin{align}
 & \sigma(\gamma_k)^2+ \frac{1}{2}(n+1)(m+1)\ell(\gamma_k)^2 \geq 2\sigma(\gamma_k)\ell(\gamma_k)\sqrt{\frac{1}{2}(n+1)(m+1)}\geq    \nonumber \\ 
 & \sigma(\gamma_k)\ell(\gamma_k)(n+1)\geq 4\ell(\gamma_k)(n+1), \nonumber
\end{align}
and from the lower bound on $n+1$, we have
\begin{align}
 & 2\sigma(\gamma_k)^2+ \frac{1}{2}(n+1)(m+1)\ell(\gamma_k)^2 \geq 2\sigma(\gamma_k)\ell(\gamma_k)\sqrt{(n+1)(m+1)}\geq    \nonumber \\ 
 & 4\ell(\gamma_k)(m+1), \nonumber
\end{align}
    and finally, we have $ \sigma(\gamma_k)^2 \geq 16\sigma(\gamma_k).$
Therefore, by adding these three inequalities, we obtain $e \geq 0$.

\item If $4(m+1)/\sigma(\gamma_k)^2\geq n+1 \geq \sigma(\gamma_k)/\ell(\gamma_k)$, then
$$
\ell(\eta_0) \geq 3\sigma(\gamma_k)-4 \geq \frac{5\sigma(\gamma_k)}{2},
$$
$$
\ell(\zeta_0) \geq \frac{4(m+1)\ell(\gamma_k)}{5}+ \frac{(m+1)\ell(\gamma_k)}{5} \geq
$$
$$
\frac{4(m+1)\ell(\gamma_k)}{5}+ \frac{(n+1)\sigma(\gamma_k)^2\ell(\gamma_k)}{20} \geq \frac{4(m+n+2)\ell(\gamma_k)}{5}.        
$$
\item If $\sigma(\gamma_k)/\ell(\gamma_k) \geq n+1$ then
$$
\ell(\eta_0) \geq 2 \sigma(\gamma_k) \, \, \, \, \, \ell(\zeta_0) \geq \sigma(\gamma_k)+(m+1)\ell(\gamma_k) \geq  (m+n+2)\ell(\gamma_k).    
$$
\end{enumerate}

\item when $\eta_0$ and $\zeta_0$ are of type one and two, respectively, we have:
$$
I_0 \leq \frac{m+1}{\ell(\eta_0)\ell(\zeta_0)},
$$
and 
$$
\ell(\eta_0) \geq 2\sigma(\gamma_k), \, \, \, \, \, \ell(\zeta_0) \geq (m+1)\ell(\gamma_k) .
$$
The case that $\eta_0$ is type two and $\zeta_0$ is type one is similar.

\item when $\eta_0$ and $\zeta_0$ are of type two, we have:
$$
I_0 \leq \frac{2\min(n,m)+2}{\ell(\eta_0)\ell(\zeta_0)},
$$
and we know

$$
\ell(\eta_0), \ell(\zeta_0) \geq 2\sqrt{(\min(n,m)+1)\ell(\gamma_k)\sigma(\gamma_k)}.
$$
\item When one of the arcs is $\gamma_k$ and the other arc is of type $1$ then we have:
$$
I_0 \leq \frac{1}{2\sigma(\gamma_k)\ell(\gamma_k)}.
$$
Note that $\gamma_k$ and an arc of type $2$ do not intersect.
\end{itemize}
\qed

We next state several consequences of Theorem \ref{theorem: main.order}.

 Recall that $I_{simple}$ is the supremum of the normalized intersection number between simple closed geodesics. We have the following:
  
 \begin{Cor}\label{cor: I_simple-asymp}
For $X\in \Mg$, we have:
$$
I_{simple}(X) \sim \frac{1}{2\sys(X)\log(1/\sys(X))},
$$
as $X \to \infty$ in $\mathcal{M}_g$. 

 \end{Cor}
\begin{proof}
  On one hand, $I_{simple}(X)\leq I(X)$, on the other hand, in Lemma \ref{lemma: geod.int.sys} the curves $\gamma_1$ and $\gamma_2$ are simple which gives a lower bound for $I_{simple}(X)$. So, from Theorem \ref{theorem: main.order}, we conclude that $I_{simple}(X)$ has the same asymptotic behavior as $I(X)$.  
\end{proof} 

 The thick part of $\mathcal{M}_g$ (the subset containing surfaces whose systole $\geq \epsilon$) is compact and $I_{simple}$ is a continuous function (similar to Lemma \ref{lemma: I.cont}). Therefore, we obtain the following result. 

 \begin{Cor}\label{cor: I_simple=I}
 For $X \in \Mg$, $I_{simple}(X)/I(X) \to 1$ as $\sys(X) \to 0$. As a result, there exists a constant $c_g>0$, depending only on the genus $g$, such that
\[
c_g\, I(X) \le I_{\mathrm{simple}}(X) \le I(X).
\]

 \end{Cor}

 From the continuity of $I(X)$ (Lemma \ref{lemma: I.cont}) and Theorem \ref{theorem: main.order} we obtain an estimate of $I(X)$ up to a multiplicative constant.

\begin{Cor}\label{cor: I.asymp.c_g}
For $X \in \Mg$, we have:
$$
I(X) \asymp_g \frac{1}{s\log(1/s)},
$$
where $s= \min(\sys(X),1/2)$.
\end{Cor}

\section{Minimum on moduli space}\label{sec: min}

In this section, we prove Proposition \ref{Prop: I.min}, which states

$$
\min \limits_{X \in \Mg} I(X) \asymp \frac{1}{(\log g)^2}, 
$$
uniformly in $g$.

First, we prove the existence of a pair of pants with short cuffs.

\begin{Lem}\label{lemma: short.pants}
Given a surface $X \in \Mg$, it has a pair of pants with cuffs of length $\leq c_0\log g$. 
\end{Lem}
\begin{proof}
 Let $\gamma_1$ be a systole of $X$. Enlarge the collar of $\gamma_1$ from one side until it runs into itself for the first time at a point $Q$. This is homeomorphic to a pair of pants. It might also be homeomorphic to a genus $0$ surface with $n$ boundaries, but the proof would be similar.  Consider the geodesic representative of its cuffs to obtain a pair of pants $P$ with geodesic boundaries $\gamma_1,\gamma_2,\gamma_3$. We know $\sys(X) \leq 2\log(4g-2)\leq 6\log(g)$ \cite[Lem. 5.2.1.]{Bsr}.
Now we show $\ell(\gamma_2),\ell(\gamma_3) \leq c_0\log g$ too.
Let $BC$ be a shortest arc between $\gamma_1$ and the seam between $\gamma_2$ and $\gamma_3$ where $B \in \gamma_1$. Assume that $AE$ is the seam between $\gamma_1,\gamma_2$ where $A\in \gamma_1$. Let $D \in \gamma_2$ be the base of the seam between $\gamma_2,\gamma_3$. See Figure \ref{fig: pentagon}.

\begin{figure}[H]
    \centering

\tikzset{every picture/.style={line width=0.75pt}} 

\begin{tikzpicture}[x=0.75pt,y=0.75pt,yscale=-1,xscale=1]

\draw   (167.24,121.09) .. controls (176.03,120.91) and (183.39,132.73) .. (183.68,147.49) .. controls (183.97,162.24) and (177.07,174.34) .. (168.28,174.51) .. controls (159.49,174.68) and (152.13,162.86) .. (151.84,148.11) .. controls (151.55,133.36) and (158.45,121.26) .. (167.24,121.09) -- cycle ;
\draw   (282.93,64.85) .. controls (293.78,62.79) and (305.28,75.36) .. (308.61,92.92) .. controls (311.93,110.48) and (305.83,126.38) .. (294.98,128.43) .. controls (284.13,130.49) and (272.63,117.92) .. (269.31,100.37) .. controls (265.98,82.81) and (272.08,66.91) .. (282.93,64.85) -- cycle ;
\draw   (296.31,153.92) .. controls (304.44,154.05) and (310.84,165.14) .. (310.62,178.68) .. controls (310.39,192.22) and (303.61,203.09) .. (295.48,202.95) .. controls (287.35,202.81) and (280.95,191.72) .. (281.18,178.18) .. controls (281.41,164.64) and (288.18,153.78) .. (296.31,153.92) -- cycle ;
\draw    (167.24,121.09) .. controls (226.73,104.59) and (247.73,81.59) .. (282.93,64.85) ;
\draw    (168.28,174.51) .. controls (210.2,165.2) and (269.2,206.2) .. (295.48,202.95) ;
\draw    (288.41,129) .. controls (271.2,128.2) and (273.2,163.2) .. (296.31,153.92) ;
\draw    (183.68,147.49) .. controls (224.2,140.2) and (258.2,145.2) .. (277.2,148.2) ;
\draw   (183.68,142.29) -- (189.17,142.29) -- (189.17,146.03) ;
\draw   (176.27,118.99) -- (177.73,122.59) -- (174.17,124.03) ;
\draw   (282.62,130.41) -- (278.53,126.33) -- (281.53,123.33) ;
\draw   (270.78,146.72) -- (271.73,141.59) -- (276.86,142.54) ;
\draw   (270.57,79.99) -- (265.73,77.59) -- (267.72,73.58) ;
\draw  [line width=3] [line join = round][line cap = round] (166.17,121.03) .. controls (166.17,120.7) and (166.17,120.36) .. (166.17,120.03) ;
\draw  [line width=3] [line join = round][line cap = round] (183.17,147.03) .. controls (183.5,147.03) and (183.84,147.03) .. (184.17,147.03) ;
\draw  [line width=3] [line join = round][line cap = round] (286.17,128.03) .. controls (286.5,128.03) and (286.84,128.03) .. (287.17,128.03) ;
\draw  [line width=3] [line join = round][line cap = round] (276.17,68.03) .. controls (275.84,68.03) and (275.5,68.03) .. (275.17,68.03) ;
\draw  [line width=3] [line join = round][line cap = round] (276.94,146.98) .. controls (276.94,147.31) and (276.94,147.64) .. (276.94,147.98) ;

\draw (159,101.4) node [anchor=north west][inner sep=0.75pt]    {$A$};
\draw (168,135.4) node [anchor=north west][inner sep=0.75pt]    {$B$};
\draw (263.62,150.81) node [anchor=north west][inner sep=0.75pt]    {$C$};
\draw (285,109.4) node [anchor=north west][inner sep=0.75pt]    {$D$};
\draw (267,50.4) node [anchor=north west][inner sep=0.75pt]    {$E$};
\draw (131,149.4) node [anchor=north west][inner sep=0.75pt]    {$ \begin{array}{l}
\gamma _{1}\\
\end{array}$};
\draw (304,71.4) node [anchor=north west][inner sep=0.75pt]    {$ \begin{array}{l}
\gamma _{2}\\
\end{array}$};
\draw (307,169.4) node [anchor=north west][inner sep=0.75pt]    {$ \begin{array}{l}
\gamma _{3}\\
\end{array}$};

\end{tikzpicture}
    \caption{}
    \label{fig: pentagon}
\end{figure}
 Now consider the right-angle pentagon $ABCDE$. It is well known that $\ell(\gamma_2)=2\ell(ED)$. From the trigonometry formula, \cite[Thm. 2.3.4.]{Bsr}, we have:
 $$
 \cosh(\ell(\gamma_2)/2)=\cosh(ED)=\sinh{\ell(AB)}\sinh{\ell(BC)}. 
 $$
 Note that $\ell(AB) \leq sys(X)$ and $\ell(BC) \leq d(Q,\gamma_1) \leq \sinh^{-1}(4\pi(g-1)/\sys(X))$ (see proof of Lemma \ref{lemma: geod.int.sys} for the last inequality).
 
 Therefore, we have:
 $$
 \cosh(\ell(\gamma_2)/2) \leq \frac{ \sinh(\sys(X))4\pi(g-1)}{\sys(X)}.
 $$
 Moreover, we can see $\sinh(x)/x\leq c_0g^2$ when $x \in (0,2\log(4g-2)]$. As a result, $\ell(\gamma_2)\leq 2\log (c_0 g^3) \leq c_0'\log g$. Similarly, we can show $\ell(\gamma_3) \leq c_0\log g$.  

\end{proof} 

A closed geodesic with exactly one self-intersection is called a \textit{figure eight}. For example, see Figure \ref{fig: fig8}.

\begin{figure}[H]
    \centering

\tikzset{every picture/.style={line width=0.75pt}} 

\begin{tikzpicture}[x=0.75pt,y=0.75pt,yscale=-1,xscale=1]

\draw [color={rgb, 255:red, 0; green, 0; blue, 0 }  ,draw opacity=1 ]   (248.39,222.51) .. controls (207.81,251.73) and (192.18,196.24) .. (148.41,220.58) ;
\draw [color={rgb, 255:red, 0; green, 0; blue, 0 }  ,draw opacity=1 ]   (148.41,220.58) .. controls (88.91,245.43) and (58.06,185.82) .. (88.74,151.41) ;
\draw [color={rgb, 255:red, 0; green, 0; blue, 0 }  ,draw opacity=1 ]   (248.39,222.51) .. controls (282.06,188.16) and (251,116) .. (188.72,153.35) ;
\draw [color={rgb, 255:red, 0; green, 0; blue, 0 }  ,draw opacity=1 ]   (235.13,184.25) .. controls (207.74,203.72) and (202.86,197.63) .. (184.1,185.26) ;
\draw [color={rgb, 255:red, 0; green, 0; blue, 0 }  ,draw opacity=1 ]   (188.72,153.35) .. controls (148.14,182.57) and (129.31,122.19) .. (88.74,151.41) ;
\draw [color={rgb, 255:red, 0; green, 0; blue, 0 }  ,draw opacity=1 ]   (227.03,189.09) .. controls (214.23,178.84) and (203.26,183.1) .. (195.09,191.94) ;
\draw [color={rgb, 255:red, 0; green, 0; blue, 0 }  ,draw opacity=1 ]   (154.01,184.58) .. controls (129.68,201.12) and (104.84,192.63) .. (96.08,180.46) ;
\draw [color={rgb, 255:red, 0; green, 0; blue, 0 }  ,draw opacity=1 ]   (141.89,190.35) .. controls (128.21,174.08) and (113.28,177.12) .. (104.09,186.94) ;
\draw  [dash pattern={on 0.84pt off 2.51pt}]  (93.2,171.2) .. controls (131.4,126.4) and (222.2,258.2) .. (248.2,196.2) ;
\draw  [dash pattern={on 0.84pt off 2.51pt}]  (93.2,171.2) .. controls (72.2,192.2) and (111.2,237.2) .. (158.2,200.2) ;
\draw  [dash pattern={on 0.84pt off 2.51pt}]  (158.2,200.2) .. controls (217.2,144.2) and (256.2,164.2) .. (248.2,196.2) ;

\end{tikzpicture}

    \caption{A figure eight closed geodesic}
    \label{fig: fig8}
\end{figure}

\begin{Lem}\label{lemma: fig8}
For $X \in \Mg$, there is a figure eight closed geodesic with length $\leq c_0\log g$.
\end{Lem}
\begin{proof}
According to Lemma \ref{lemma: short.pants}, there is a pair of pants $P$ with cuffs $\gamma_1,\gamma_2,\gamma_3$ of length $\leq c_0\log g$. 
Let $\gamma$ be a figure eight geodesic in $P$. Then, using trigonometry in hyperbolic geometry \cite[Eqn. 4.2.3.]{Bsr} we have:
$$
\cosh(\frac{\ell(\gamma)}{2})=\cosh(\frac{\ell(\gamma_3)}{2})+2\cosh(\frac{\ell(\gamma_1)}{2})\cosh(\frac{\ell(\gamma_2)}{2}) \leq c_0g^{2c_0'}
$$

As a result, $\ell(\gamma)\leq c\log(g)$ for a constant $c>0$. 
\end{proof}

\textbf{Proof of Proposition \ref{Prop: I.min}.}
It is known that there is a hyperbolic surface $X_0 \in \Mg$ such that $\sys(X_0) \geq c_0\log g$ \cite{sys}\cite{buser-sarnak}. Therefore, by Proposition \ref{prop: int.bnd} we have $I(X_0)\leq b/(\log g)^2$ where $b=4/c_0^2$.  
On the other hand, by Lemma \ref{lemma: fig8}, we obtain $I(X)\geq 1/(c\log g)^2$ for all $X \in \Mg$. 
\qed

\section{Finite volume surfaces with cusps}\label{sec: cusp}

In this section, we prove Theorem \ref{thrm: finite.vol}.

Consider a proper pants decomposition of $X$ where pairs of pants may have cusps instead of geodesic boundaries. Similar to the compact case, we can define thin and thick decomposition. For a cusp $c$, define the thin part $O(thin, c)$ as the horoball neighborhood of $c$ whose horocycle boundary has length $1$. Note that the length of a thin part's boundary tends to $1$ while the core closed geodesic's length converges to $0$.
Similarly, we define the collar $O(collar,c)$ of $c$ as the horoball neighborhood of $c$ with the horocycle boundary of length $2$. We can obtain results similar to those in the compact case. However, we do not repeat the proofs here and we only state the results.
\begin{itemize}
    \item simple geodesic arcs with endpoints on the horocycle of length $2$ are disjoint from the thin part of the cusp (similar to Lemma \ref{lemma: has.int}),
    \item the distance between two cuffs of a thick part is $\geq c_g$ since the distance between the horocycles of length $1$ and $2$ is $\log 2$,
    \item the distances between the seams of a thick part are $\geq c_g'$.
\end{itemize}  

Let $\gamma_n$ be a geodesic arc with endpoints on the horocycle of length $2$ and with a winding number $n$ around the cusp. We can see:
\begin{itemize}
    \item $\ell(\gamma_n) \in [2\log n,2\log n+\log 24]$ using the distance formula \cite[Equ. 1.1.2]{Bsr},
    \item the intersection number $i(\gamma_n,\gamma_m)$ between $\gamma_n$ and $\gamma_m$ is in $[s,s+2]$ where $s=2\min(n,m)$ (similar to Lemma \ref{lemma: int}),
    \item $n \leq 2/r$ when $\gamma_n$ does not enter the horoball with the horocycle boundary of length $r$. Moreover, when $n \leq (2/r)-1$, $\gamma_n$ does not enter the horoball whose boundary has length $r$.
\end{itemize}

\textbf{Proof of Theorem \ref{thrm: finite.vol}.} Consider a proper thin and thick decomposition of $X$. The interaction strength inside the thick parts is $\leq c_g$ (similar to Lemma \ref{lemma: I.thick}). When a thin part is the neighborhood of a closed geodesic, then, by the proof of Theorem \ref{theorem: main.order}, we can see that the interaction strength is asymptotically $1/(2\sys(X)\log(1/\sys(X)))$ as $\sys(X) \to 0$. Otherwise, when a thin part is the neighborhood of a cusp $c$, we can see for $n\leq m \leq 2/r$:
$$
\frac{i_{O(thin,c)}(\gamma_n,\gamma_m)}{\ell_{O(collar, c)}(\gamma_n)\ell_{O(collar,c)}(\gamma_m)}\leq \frac{2n+2}{4(\log n)^2}.
$$
Therefore, the interaction strength in $O(thin, c)$ is asymptotically $\leq 1/(r(\log 1/r)^2)$. This gives the proper upper bound for $I(X_r)$. On the other hand, the left-hand side ratio for $n=m=\lfloor (2/r)-1 \rfloor$ is asymptotically 
$$
\frac{2n}{(2\log(n)+\log 24)^2} \geq \frac{1}{r(\log \frac{1}{r})^2},
$$
which gives the required lower bound for $I(X_r)$ as $r \to 0$. \qed

\vspace{0.5em}
The subset of $\mathcal{M}_{g,n}$ including hyperbolic surfaces $X\in \mathcal{M}_{g,n}$ with $\sys(X) \geq \epsilon$ is compact. Therefore, we have:

\begin{Cor}\label{cor: finite.vol}
Fix $g\geq 2, n>0$ and let $X$ be a finite-volume hyperbolic surface of genus $g$ with $n$ cusps. Then we have:
$$
I(X_r) \asymp_{g,n} \max \left( \frac{1}{s\log(1/s)}, \frac{1}{r(\log(1/r))^2}\right)
 $$
 where $s=\min(\sys(X),1/2)$.
\end{Cor}

\section{Expected value on moduli space}\label{sec: random}
In this section, we prove Corollary \ref{cor: exp} which states:\\

\emph{The expected value of the interaction strength on} $\mathcal{M}_g,$ \emph{equipped with the Weil-Petersson measure, is finite.}

A volume form $\volwp$ on $\Mg$ is associated with the Weil-Petersson symplectic form on $\mathcal{T}_g$. This induces a probability measure $\mathbb{P}_g$ defined as the following:
$$
\mathbb{P}_g(A):= \frac{\volwp(A)}{\volwp(\Mg)}.
$$
Define $\mathcal{M}^{\epsilon}_g:= \{X \in \Mg | \sys(X) \leq \epsilon \}$. Mirzakhani showed that there is a constant $C_0$ such that  
\begin{equation}\label{equation: exp}
    \frac{1}{C_0} \epsilon^2 \leq \mathbb{P}_g(\mathcal{M}^{\epsilon}_g) \leq C_0 \epsilon^2,  
\end{equation}
for $g\geq 2$ \cite[\S5]{Mir.sys}\cite{Mir.WP}.\\

\textbf{Proof of Corollary \ref{cor: exp}.} 
The complement of $\mathcal{M}_g^{\epsilon}$ is compact. Therefore, it is enough to show that the integral on $\mathcal{M}_g^{\epsilon}$ is finite. Define $\mathcal{M}^{[a,b]}_g:= \{X \in \Mg |\, \sys(X) \in [a,b]\}$. From Equation (\ref{equation: exp}), we have: 
$$
 \int \limits_{\mathcal{M}^{[2^{-(n+1)},2^{-n}]}_g} \frac{1}{\sys(X)} \, d \volwp \leq \frac{c_0}{2^{n-1}}.
$$
Therefore, from Theorem \ref{theorem: main.order}, if $\epsilon \leq 1/2^{n_0}$ then
$$
\int \limits_{\mathcal{M}_g^{\epsilon}} I(X) \, d \volwp \leq \int \limits_{\mathcal{M}^{\epsilon}_g} \frac{1}{\sys(X)} \, d \volwp \leq \sum \limits_{m\geq n_0} \frac{c_0}{2^{m-1}}=\frac{c_0}{2^{n_0-2}},
$$
which is finite, as required.\qed 

  \vspace{0.5em}
  Note that from (\ref{equation: exp}) we know 
$$
\int \limits_{\mathcal{M}^{\epsilon}_g} \frac{1}{\sys(X)^2} \, d \volwp \geq \frac{1}{C_0},
$$
for $\epsilon>0$ small enough. Therefore, $\int \limits_{\Mg} \frac{1}{\sys(X)^2} d \volwp=\infty$. Consequently, the upper bound $4/\sys(X)^2$ on $I(X)$ is not sufficient to conclude Corollary \ref{cor: I.asymp.c_g}.

\bibliography{biblio}
\bibliographystyle{math}

\end{document}